\documentclass{birkjour}

\usepackage[utf8]{inputenc}
\usepackage[T1]{fontenc}
\usepackage[english]{babel}
\usepackage{mathtools,amsmath,amssymb,amsfonts,amsthm,mathrsfs}
\mathtoolsset{showonlyrefs}
\usepackage{enumitem}

\usepackage{geometry}
\usepackage{graphicx,color}
\usepackage{tikz,pgfplots,pgf}
\usepackage[font=small]{caption}

\usepackage[bookmarksnumbered,colorlinks=true,urlcolor=black,linkcolor=black,citecolor=black,
bookmarks, breaklinks]{hyperref}
\makeatletter
\newcommand{\leqnomode}{\tagsleft@true\let\veqno\@@leqno}
\makeatother

\newtheorem{theorem}{Theorem}[section]
\newtheorem{lemma}{Lemma}[section]
\newtheorem{proposition}{Proposition}[section]

\theoremstyle{definition}

\theoremstyle{definition}
\newtheorem{remark}{Remark}[section]

\numberwithin{equation}{section}

\begin{document}

\title[Periodic solutions to superlinear indefinite planar systems]{Periodic solutions to \\superlinear indefinite planar systems: \\ a topological degree approach}

\author[G.~Feltrin]{Guglielmo Feltrin}

\address{
Department of Mathematics, Computer Science and Physics, University of Udine\\
Via delle Scienze 206, 33100 Udine, Italy}

\email{guglielmo.feltrin@uniud.it}

\author[J.~C.~Sampedro]{Juan Carlos Sampedro}

\address{
Institute of Interdisciplinary Mathematics,\\
Department of Mathematical Analysis and Applied Mathematics, \\
Complutense University of Madrid,\\
Plaza de las Ciencias 3, 28040 Madrid, Spain}

\email{juancsam@ucm.es}

\author[F.~Zanolin]{Fabio Zanolin}

\address{
Department of Mathematics, Computer Science and Physics, University of Udine\\
Via delle Scienze 206, 33100 Udine, Italy}

\email{fabio.zanolin@uniud.it}

\thanks{Work performed under the auspices of the Grup\-po Na\-zio\-na\-le per l'Anali\-si Ma\-te\-ma\-ti\-ca, la Pro\-ba\-bi\-li\-t\`{a} e le lo\-ro Appli\-ca\-zio\-ni (GNAMPA) of the Isti\-tu\-to Na\-zio\-na\-le di Al\-ta Ma\-te\-ma\-ti\-ca (INdAM). The first author is grateful to Prof.~Juli\'{a}n L\'{o}pez-G\'{o}mez and to the Complutense University of Madrid for the kind hospitality during the period in which part of the work has been written. The second author is grateful to the other authors and to the University of Udine for allowing him to perform his predoctoral stay there. He has been supported by the grants PRE2019\_1\_0220 and EP\_2022\_1\_0063 of the Basque Country Government. The results have been presented by the third author in the ``International Workshop on Nonlinear Differential Problems'', Taormina-Giardini Naxos, September 19--22, 2022.\\
\textbf{Preprint -- November 2022}}



\dedicatory{}

\date{}

\begin{abstract}
We deal with a planar differential system of the form
\begin{equation*}
\begin{cases}
\, u' = h(t,v),
\\
\, v' = - \lambda a(t) g(u),
\end{cases}
\end{equation*}
where $h$ is $T$-periodic in the first variable and strictly increasing in the second variable, $\lambda>0$, $a$ is a sign-changing $T$-periodic weight function and $g$ is superlinear. Based on the coincidence degree theory, in dependence of $\lambda$, we prove the existence of $T$-periodic solutions $(u,v)$ such that $u(t)>0$ for all $t\in\mathbb{R}$. Our results generalize and unify previous contributions about Butler's problem on positive periodic solutions for second-order differential equations (involving linear or $\phi$-Laplacian-type differential operators).
\end{abstract}

\maketitle

\section{Introduction}\label{section-1}

In this paper, we investigate the problem of existence of $T$-periodic solutions $(u,v)$, with $u(t)>0$ for all $t\in\mathbb{R}$, of the planar system
\begin{equation}\label{syst-intro-k}
\begin{cases}
\, u' = h(t,v),
\\
\, v' = k(t,u),
\end{cases}
\end{equation}
where $h,k\colon\mathbb{R}\times\mathbb{R}\to\mathbb{R}$ are $T$-periodic functions in the first variable.

The main motivation for this problem comes from the study of positive $T$-periodic solutions for the strongly nonlinear scalar differential equation
\begin{equation}\label{eq-LF}
\mathscr{L} u = \mathscr{F}(t,u),
\end{equation}
where $\mathscr{L}$ is a nonlinear differential operator of the form $\mathscr{L} u := -(\Phi(t,u'))'$ and $\mathscr{F}$ is $T$-periodic in the first variable. When $\Phi(t,s)=\xi$ is invertible to $s = h(t,\xi)$, then we can rewrite equation \eqref{eq-LF} equivalently as
\begin{equation}\label{syst-LF}
\begin{cases}
\, u' = h(t,v),
\\
\, v' = - \mathscr{F}(t,u).
\end{cases}
\end{equation}
In this article we are interested in nonlinearities of the form
\begin{equation}\label{eq-ag-intro}
\mathscr{F}(t,u)=a(t)g(u),
\end{equation}
where $a\colon\mathbb{R}\to\mathbb{R}$ is a $T$-periodic function and $g\colon\mathopen{[}0,+\infty\mathclose{[}\to\mathopen{[}0,+\infty\mathclose{[}$ is such that
\begin{equation*}
g(0)=0, \qquad g(s)>0, \; \text{ for every $s\in\mathopen{]}0,+\infty\mathclose{[}$.}
\end{equation*}
If we look for non-trivial $T$-periodic solutions $(u,v)$ of \eqref{syst-LF}, then 
\begin{equation*}
\int_{0}^{T} \mathscr{F}(t,u(t))\,\mathrm{d}t=0
\end{equation*}
and thus in the special case \eqref{eq-ag-intro} a necessary condition for the existence of a $T$-periodic solution $(u,v)$ of \eqref{syst-LF} which is positive in the $u$-component is that the weight $a$ changes its sign.
Accordingly to a terminology which is of common use after \cite{HeKa-1980}, the associated periodic boundary value problem enters in the class of \textit{indefinite problems}. In the last decades this kind of differential problems has been widely studied, both in the ODE and in the PDE setting. 
Concerning the case of second-order differential equations with $\mathscr{L}(u)=-u''$ (or $\mathscr{L}(u)=-\Delta u$) we refer the reader to the classical papers \cite{AlTa-1996, AmLG-1998, BaPoTe-1988, BeCDNi-1994, BeCDNi-1995,GRLG-2000,LG-1997,LG-2000}, till to very recent contributions \cite{ClSz-2022, FeGi-2020, FeSoTe-2022, FeLG-2022, KaRQUm-2021, LGOm-2022, MaReZa-2021,Te-2020} and the references therein; see also \cite{Fe-2018} for a wide bibliography on the subject.

One of the first attempts in the study of the periodic solutions to the second-order superlinear equation with indefinite weight $-u'' = a(t) g(u)$ was given by G.~J.~Butler in \cite{Bu-1976}. 
Besides proving the existence of infinitely many $T$-periodic solutions which oscillate, Butler also raised the question of the existence of non-oscillatory/positive $T$-periodic solutions when $g(s)=|s|^{\alpha}\mathrm{sign}(s)$ for $\alpha>1$ and $\int_{0}^{T} a(t)\,\mathrm{d}t<0$.
Such a problem was solved in the affirmative in \cite{FeZa-2015} using a topological degree approach, namely Mawhin's coincidence degree. Previous solutions to Butler's conjecture for the Neumann problem can be found in \cite{BeCDNi-1994,BeCDNi-1995}.

There is a large literature concerning the existence of nontrivial solutions to superlinear Hamiltonian systems with sign-changing coefficients, see for instance \cite{BoLiSc-2017,TaWu-2004} and the references therein. However, fewer results about positive solutions are available for equations generalizing Butler's model. Other works in this area, like \cite{GKW-2008,To-2003}, provided the existence of positive periodic solutions for nonlinear perturbation of the linear differential operator $\mathscr{L}(u)=-u''+\rho(t)u$, by exploiting the properties of the associated Green's function and fixed point theorems on positive cones of Banach spaces. However, we stress that for the Butler's example the geometry is completely different.

The aim of this article is to show the effectiveness of the coincidence degree approach to study the planar system \eqref{syst-intro-k} in order to generalize and unify some previous results to a wide class of quasi-linear ODEs of the form \eqref{eq-LF}.
In doing so, we will also introduce a new concept of \textit{superlinearity} at zero and at infinity for planar systems which generalizes the previously known conditions when $\mathscr{L}$ is a $\phi$-Laplacian operator.

When \eqref{eq-LF} takes the form
\begin{equation*}
-(\phi(u'))'=\mathscr{F}(t,u),
\end{equation*}
where $\phi\colon I \to J$ is an increasing homeomorphism between two open intervals $I$ and $J$,
a topological framework has been proposed in \cite{BeDOMe-2006,BeMa-2007,GHMaZa-1993,MaMa-1998}.
This requires a preliminary analysis of the strongly nonlinear equation $-(\phi(u'))'= h(t)$ in suitable function spaces.
Here, we present a related but slightly different approach which consists in inverting $\phi$ on the real line and subsequently in studying the associated planar system \eqref{syst-LF} for $h(t,v):=\phi^{-1}(v)$.
Clearly, in this manner, we can set in the framework of system \eqref{syst-LF} also more general nonlinear differential operators of $\phi$-Laplacian type.
Such an approach was already considered in \cite{FeZa-2017} and, more recently, in \cite{DoZa-2020,MaReZa-2021} with respect to the study of oscillatory solutions, and in \cite{FoKlSf-2021,FoTo-2021} to generalize the concept of upper and lower solutions in the frame of planar systems.

Following the above premises and motivations, we are lead now to study the problem of positive $T$-periodic solutions to the planar system \eqref{syst-LF} with $h(t,\cdot)$ a strictly increasing function (as it will be the inverse of a suitable homeomorphism $\phi$). More precisely, recalling \eqref{eq-LF} and \eqref{eq-ag-intro}, we consider the parameter-dependent system
\begin{equation}\label{syst-intro}
\begin{cases}
\, u' = h(t,v),
\\
\, v' = -\lambda a(t)g(u).
\end{cases}
\end{equation}
The choice to introduce a parameter $\lambda>0$ is made convenient in order to present some general results for a broader class of $h$ and $g$. This will be clear from the statement of the main results. Moreover, it will be also useful to embed \eqref{syst-intro} into a system of the form
\begin{equation}\label{syst-intro-f}
\begin{cases}
\, u' = h(t,v),
\\
\, v' = - f(t,u),
\end{cases}
\end{equation}
with $f(t,u)=\lambda a(t)g(u)$ for $u\geq0$ and $f(t,u)<0$ for $u<0$.

The strategy to find the periodic solutions $(u,v)$ of \eqref{syst-intro} with $u>0$ is based on the following steps.
\begin{enumerate}[leftmargin=65pt,labelsep=8pt]
\item[\textit{First step.}] Transform the periodic problem associated with system \eqref{syst-intro-f} into an equivalent fixed point problem in function spaces.
\item[\textit{Second step.}] Prove that there is a nonzero degree on a small ball centered at the origin.
\item[\textit{Third step.}] Prove that the degree is zero on a larger ball centered at the origin.
\item[\textit{Forth step.}] Prove that a non-trivial $T$-periodic solution $(u,v)$ of system \eqref{syst-intro-f} is positive in the $u$-component.
\end{enumerate}

The first step is provided by an application of coincidence degree theory, by transforming the $T$-periodic problem for system \eqref{syst-intro-f} into an equivalent operator equation
\begin{equation*}
L z = N z, \quad z=(u,v)\in X,
\end{equation*}
where $X=\mathcal{C}_{T}$ is the Banach space of continuous $T$-periodic functions with range in $\mathbb{R}^{2}$. See Section~\ref{section-3} for the details.

The second and the third step follow a classical technique introduced by Nussbaum \cite{Nu-1975}, see also \cite{Am-1976, deF-1982}, which is commonly used in the search of positive solutions using the Krasnosel'skii--Amann approach of nonlinear operators defined in positive cones in Banach spaces.
In our case we do not apply the theory of positive operators, since it is sufficient to produce a non-trivial $T$-periodic solution $(u,v)$. 
In more detail, for the second step we rely on an approach similar to the one considered in \cite{FeZa-2015} which requires some mild extra assumptions on $g$ in a right neighborhood of the origin (like smoothness or a condition of ``slowly oscillation''). We also need to introduce a condition of ``superlinear/below linear'' growth at zero which extends the classical hypothesis $g(s)/s\to 0$ as $s\to 0^{+}$.
More specifically, our assumption reads as
\begin{equation}\label{main-cond-zero-INTRO}
\lim_{s\to 0^{+}} \dfrac{h(t, K g(s))}{s} = 0,
\quad \text{uniformly in $t$, for every $K\in\mathbb{R}$.}
\end{equation}
The third step appears to be more delicate from the technical point of view because, due to the general nature of $h$, some eigenvalue-type estimates obtained in \cite{FeZa-2015} are no more available. Nonetheless we can provide the required estimates into two different cases: $\lambda>0$ large (this allows to cover also the case in which $h$ is bounded) and $\lambda>0$ arbitrary, by adding a new condition of superlinear growth at infinity, which extends the classical hypothesis $g(s)/s\to +\infty$ as $s\to +\infty$. Such new condition reads as
\begin{equation}\label{main-cond-infinity-INTRO}
\lim_{s\to+\infty} \mathrm{sign}(K) \frac{h(t,Kg(s))}{s} = +\infty, 
\quad \text{uniformly in $t$, for every $K\in\mathbb{R}\setminus\{0\}$.}
\end{equation}
The abstract degree framework for the second and third steps is illustrated in Section~\ref{section-3}, while the technical estimates are given in Section~\ref{section-4}.

As a consequence of the second and third step, and the additivity property of the degree, we deduce that the degree is non zero on an annular domain, thus a non-trivial $T$-periodic solution $(u,v)$ of the planar system \eqref{syst-intro-f} exists.

Finally, the crucial step is the last one, which requires to check that $u(t)>0$ for all $t$. For this we need a maximum principle for \eqref{syst-intro-f} suitably adapted to the case of periodic boundary value problem. The weak form of the maximum principle (leading to $u(t)\geq0$ for all $t\in\mathbb{R}$) is almost obvious, while the strong form requires a more delicate analysis.
We propose two versions of such strong maximum principle: one which fits with the case $a\in L^{1}$ and extends to system \eqref{syst-intro-f} a previous result from \cite{MaNjZa-1995} (see Proposition~\ref{strong-max-p}), and a second sharper version (along the line of \cite{PuSe-2007} and \cite{LGOm-2023}) which deals better with $a\in L^{\infty}$ (see Proposition~\ref{strong-max-p-2}). We believe that such results may have some independent interest also with respect to possible future applications to different boundary value problems for planar systems. However, we have collected these results in the appendix since they appear to be of auxiliary type and are not related to the main core of the present paper (see also Section~\ref{section-2.1}).

As a consequence of the above procedure, we can state and prove the following result.

\begin{theorem}\label{th-main-INTRO}
Let $h\colon\mathbb{R}\times\mathbb{R}\to\mathbb{R}$ be a continuous function, $T$-periodic in the first variable, with $h(t,0)\equiv0$, and such that $h(t,\cdot)$ is strictly monotone increasing.
Let $a\colon\mathbb{R}\to\mathbb{R}$ be a sign-changing (piecewise) continuous $T$-periodic function with a finite number of zeros such that $\int_{0}^{T} a(t)\,\mathrm{d}t<0$. Let $g\colon\mathopen{[}0,+\infty\mathclose{[}\to\mathopen{[}0,+\infty\mathclose{[}$ be a continuous function, regularly oscillating at zero, with $g(0)=0$, $g(s)>0$ for $s>0$.
Then, the following conclusions hold.
\begin{itemize}
\item[(i)] If the pair $(h,g)$ satisfies \eqref{main-cond-zero-INTRO}, then system \eqref{syst-intro} has at least one $T$-periodic solution $(u,v)$ with $u>0$, provided that $\lambda>0$ is sufficiently large.
\item[(ii)] If the pair $(h,g)$ satisfies \eqref{main-cond-zero-INTRO} and \eqref{main-cond-infinity-INTRO}, then system \eqref{syst-intro} has at least one $T$-periodic solution $(u,v)$ with $u>0$, for every $\lambda>0$.
\end{itemize}
\end{theorem}

Theorem~\ref{th-main-INTRO} is a simplified version of our main results in Section~\ref{section-2} (see Theorem~\ref{th-main} and Theorem~\ref{th-main-2} for the general statements). Actually, we can assume $a\in L^{1}_{\mathrm{loc}}$ and such that it even eventually vanish in some subintervals of its domain. The precise technical sign hypothesis on $a$ is given in \ref{cond-a-star} in Section~\ref{section-2}. 

The assumption of regular oscillation of $g$ at zero is always satisfied when $g(s)$ has a power growth as $s\to0^{+}$. Such condition also includes many other possible behaviours of $g(s)$ as $s\to 0^{+}$. Regular oscillation hypotheses are commonly found in the qualitative theory of ODEs and are related to the famous Karamata-type conditions (cf.~\cite{Se-1976} and the references therein).
Alternatively we can also impose a suitable smoothness hypothesis for $g$ in a neighborhood of zero (see Remark~\ref{rem-2.2}).

In order to better clarify the meaning of our generalized superlinear conditions at zero and at infinity, let us consider the special case
\begin{equation}\label{syst-intro-alphabeta}
\begin{cases}
\, u' = b(t) |v|^{\alpha-1} v,
\\
\, v' = - a(t) |u|^{\beta-1} u,
\end{cases}
\end{equation}
with $\alpha,\beta>0$ and $b(t)>0$ for all $t\in\mathbb{R}$. Then, both \eqref{main-cond-zero-INTRO} and \eqref{main-cond-infinity-INTRO} are satisfied provided that $\alpha \beta >1$.
Applying the results in \cite{CiGaMa-2000}, one can check that the condition $\alpha \beta >1$ implies that the autonomous planar system
\begin{equation}\label{syst-intro-alphabeta-aut}
\begin{cases}
\, u' = B |v|^{\alpha-1} v,
\\
\, v' = - A |u|^{\beta-1} u,
\end{cases}
\end{equation}
with $A,B$ positive real numbers, defines a global center at the origin with the period map $\tau(E)$ of the orbits at energy level $E$ satisfying
\begin{equation*}
\lim_{E\to 0} \tau(E) = +\infty, \qquad \lim_{E\to +\infty} \tau(E) = 0,
\end{equation*}
which is the typical dynamical property associated with the superlinear systems. Observe also that for \eqref{syst-intro-alphabeta} we have $g(s)=s^{\beta}$ for $s>0$, so that $g$ is strictly increasing and therefore the average hypothesis $\int_{0}^{T} a(t)\,\mathrm{d}t<0$ turns out to be a necessary condition for the existence of solutions positive in the $u$-component. Hence, the example given by \eqref{syst-intro-alphabeta}, although elementary, shows that our assumptions in Theorem~\ref{th-main-INTRO} are sharp.
We also remark that the Hamiltonian associated to \eqref{syst-intro-alphabeta-aut} is
\begin{equation*}
H(u,v) = A \dfrac{|u|^{\beta+1}}{\beta+1} + B \dfrac{|v|^{\alpha+1}}{\alpha+1}
\end{equation*}
and, according to the terminology used in \cite{YaLo-2007}, it is called $(q,p)$-quasi homogenous of quasi-degree $pq$ for $p=\beta+1$, $q=\alpha+1$, due to the fact that $H(\lambda^{q}u, \lambda^{p}v) = \lambda^{pq} H(u,v)$ for every $\lambda>0$. Notice that the superlinear condition $\mu:=1-p^{-1}-q^{-1}>0$ considered in \cite[p.~1039]{FaFo-2022} is equivalent to our assumption $\alpha\beta>1$. For system \eqref{syst-intro-alphabeta} conditions for stability/instability of the origin were studied in \cite{Li-2006}, under additional hypotheses on $\alpha,\beta$ and on the weight coefficients $a,b$. It is interesting to mention that our condition $b(t)>0$ for every $t$ and $\int_{0}^{T} a(t)\,\mathrm{d}t<0$ are consistent with the instability result obtained in \cite{Li-2006}.

The main existence theorems for system \eqref{syst-intro} allow an immediate application to $\phi$-Laplacian differential equations of the form
\begin{equation}\label{eq-intro-phi}
(\phi(u'))'+\lambda a(t)g(u)=0.
\end{equation}
Here, possible different kind of $\phi$ are given by the $p$-Laplacian $\phi(s) = \phi_{p}(s) := |s|^{p-2}s$ (with $p>1$), the $(p,q)$-Laplacian $\phi(s) = |s|^{p-2}s+|s|^{q-2}s$ (with $1< q < p < +\infty$), or the differential operator associated with the relativistic acceleration
\begin{equation*}
\phi(s) = \dfrac{s}{\sqrt{1-s^{2}}},
\end{equation*}
just to mention a few more commonly studied examples (see \cite{BaDAPa-2016,BoFe-20,BoFeZa-2021,FeSoZa-2019,GHMaZa-1997,GHMaZa-2011,KaKnKw-1991,MaMa-1998,ToZa-2021} for previous studies on this kind of equations).
The fact that we allow a time-dependence in the function $h$ permits to consider differential operators
\begin{equation*}
u \mapsto -(\phi(t,u'))'
\end{equation*}
and consequently we can deal also with $p(t)$-Laplacian equations like
\begin{equation*}
(|u'|^{p(t)-2}u')'+\lambda a(t)g(u)=0
\end{equation*}
(see also \cite{BaCh-2013,BCMP-2021,FaFa-2003} for recent contributions in this framework).
A non-exhaustive choice of these examples is given in Section~\ref{section-5}, just to show the wide range of applicability of our results. In any case, it may be interesting to observe that the new superlinear conditions at zero and at infinity can be expressed in the setting of equation \eqref{eq-intro-phi} as
\begin{equation*}
\limsup_{s\to 0} \dfrac{g(s)}{\phi(s)} =0, 
\qquad
\limsup_{s\to +\infty} \dfrac{g(s)}{\phi(s)} = +\infty, 
\end{equation*}
for $\phi\colon \mathbb{R} \to \mathbb{R}$ an odd increasing homeomorphism satisfying the $\sigma$-conditions at zero and at infinity (cf.~\cite{GHMaZa-2011}). We refer to Section~\ref{section-5} for all the applications and details.

As a final remark, we observe that even if we investigate system \eqref{syst-intro-f} which is a planar Hamiltonian system, we stress that our method based on the topological degree theory applies to non-Hamiltonian systems as well. Actually, as in \cite{FeZa-2015}, we could deal with the planar system
\begin{equation*}
\begin{cases}
\, u' = h(t,u,v),
\\
\, v' = - f(t,u),
\end{cases}
\end{equation*}
with $h(t,u,\cdot)$ having the same kind of properties as $h(t,\cdot)$ in system \eqref{syst-intro-f}. In order to avoid unnecessarily burdening of the technical parts, we do not examine this general problem.

As a final observation, we stress that all the results of this paper extend to the Neumann boundary value problem (see Remark~\ref{rem-2.3}). We do not consider here this case in order to avoid unnecessary repetitions.

\section{Hypotheses and statement of the main result}\label{section-2}

In this section, we introduce the differential problem we deal with and we state the main result obtained. We consider the planar system
\begin{equation}\leqnomode
\label{eq:system-S}
\begin{cases}
\, u' = h(t,v),
\\
\, v' = -\lambda a(t)g(u).
\end{cases}
\tag{$\mathcal{S}$}
\end{equation}
We proceed by listing some technical assumptions on the functions $h$, $g$, and $a$.

\medskip

\noindent
\textbf{Hypotheses on $h$.}
We assume that $h\colon\mathbb{R}\times\mathbb{R}\to\mathbb{R}$ is continuous, $T$-periodic in the first variable, and such that
\begin{enumerate}[leftmargin=28pt,labelsep=10pt,label=\textup{$(h_{0})$}]
\item $h(t,0)=0$, for every $t\in\mathbb{R}$;
\label{cond-h-0}
\end{enumerate}
\begin{enumerate}[leftmargin=28pt,labelsep=10pt,label=\textup{$(h_{1})$}]
\item for every $t\in\mathbb{R}$, the function $s \mapsto h(t,s)$ is strictly increasing.
\label{cond-h-2}
\end{enumerate}
As a consequence of \ref{cond-h-0} and \ref{cond-h-2}, we deduce that
\begin{equation}\label{cond-h-1}
h(t,s) s >0, \quad \text{for every $t\in\mathbb{R}$ and $s\in\mathbb{R}\setminus\{0\}$.}
\end{equation}
For future consideration, thanks to the $T$-periodicity of $h$ in the $t$-variable, we fix two strictly increasing functions $\underline{h},\overline{h} \colon \mathbb{R} \to \mathbb{R}$ such that 
\begin{equation*}
\underline{h}(0)=\overline{h}(0)=0, \qquad 
0 \leq \underline{h}(s)s \leq h(t,s)s \leq \overline{h}(s)s, 
\quad \text{for all $s\in\mathbb{R}$.}
\end{equation*}
Accordingly, we define
\begin{equation*}
\underline{h}(s):=\min_{t\in\mathopen{[}0,T\mathclose{]}}h(t,s),
\quad
\overline{h}(s):=\max_{t\in\mathopen{[}0,T\mathclose{]}}h(t,s),
\quad
\text{for $s\geq 0$,}
\end{equation*}
and
\begin{equation*}
\underline{h}(s):=\max_{t\in\mathopen{[}0,T\mathclose{]}}h(t,s),
\quad
\overline{h}(s):=\min_{t\in\mathopen{[}0,T\mathclose{]}}h(t,s),
\quad
\text{for $s< 0$.}
\end{equation*}
Incidentally, notice that the conditions on $h$ are quite natural if we think that $h$ can be considered as the inverse of an invertible map $\phi(t,\cdot)$.

\medskip

\noindent
\textbf{Hypotheses on $g$.}
We assume that $g\colon\mathopen{[}0,+\infty\mathclose{[}\to\mathopen{[}0,+\infty\mathclose{[}$ is continuous and such that
\begin{equation}\leqnomode
\label{eq:cond-g-star}
g(0)=0, \qquad g(s)>0, \; \text{ for every $s\in\mathopen{]}0,+\infty\mathclose{[}$.}
\tag{$g_{*}$}
\end{equation}
Moreover, we introduce the function $\overline{g}\colon \mathopen{]}0,+\infty\mathclose{[} \to \mathopen{]}0,+\infty\mathclose{[}$ defined by
\begin{equation*}
\overline{g}(s) := \max_{\xi\in\mathopen{[}0,s\mathclose{]}} g(\xi).
\end{equation*}
Similarly, we introduce the function $\underline{g}\colon \mathopen{]}0,+\infty\mathclose{[} \to \mathopen{]}0,+\infty\mathclose{[}$ defined by
\begin{equation*}
\underline{g}(s) := \min_{\xi\in\mathopen{[}\frac{s}{2},s\mathclose{]}} g(\xi).
\end{equation*}
We also recall that a function $g\colon\mathopen{[}0,+\infty\mathclose{[}\to \mathopen{[}0,+\infty\mathclose{[}$ satisfying \eqref{eq:cond-g-star} is said to be \textit{regularly oscillating at zero} if
\begin{equation*}
\lim_{\substack{s\to 0^{+}\\\omega\to 1}} \dfrac{g(\omega s)}{g(s)} = 1.
\end{equation*}
We refer the reader to \cite{BGT-1987,FeZa-2015,Se-1976} for a discussion on this concept and its relevance in real analysis.

\medskip

\noindent
\textbf{Hypotheses on $a$.}
We assume that $a\colon\mathbb{R}\to\mathbb{R}$ is a locally integrable $T$-periodic function such that
\begin{equation}\leqnomode
\label{eq:cond-a-diesis}
\int_{0}^{T} a(t) \,\mathrm{d}t <0.
\tag{$a_{\#}$}
\end{equation}
Furthermore, we assume that
\begin{enumerate}[leftmargin=27pt,labelsep=10pt,label=\textup{$(a_{*})$}]
\item there exist $N \geq 1$ closed and pairwise disjoint intervals $J_{n}$ in the quotient space $\mathbb{R}/T\mathbb{Z}$ such that
\begin{align*}
\qquad\qquad a>0, \; \text{ for a.e.~$t\in J_{n}$,}
\end{align*}
for every $n=1,\ldots,N$.
\label{cond-a-star}
\end{enumerate}
As a consequence, $a\leq 0$ for a.e.~$t\in (\mathbb{R}/T\mathbb{Z}) \setminus \bigcup_{n=1}^{N} J_{n}$. To simplify the exposition, we set
\begin{equation*}
J_{n}= \mathopen{[} \sigma_{n},\tau_{n} \mathclose{]},\quad  n=1,\ldots,N,
\end{equation*}
where $\sigma_{1}<\tau_{1}<\sigma_{2}<\tau_{2}<\ldots<\sigma_{N}<\tau_{N}<\sigma_{N+1} = \sigma_{1} + T$. 

For future consideration, we introduce the constant
\begin{equation}\label{def-gamma}
\gamma := \min_{n=1,\ldots,N} |J_{n}|
\end{equation}
and, for each $\delta\in\mathopen{]}0,\gamma\mathclose{]}$, the function
\begin{equation}\label{def-A-star}
A^{*}(\delta):= \inf \biggl{\{} \int_{J} a(t)\,\mathrm{d}t \colon J \text{\,interval}, \textstyle J\subseteq \bigcup_{n=1}^{N}J_{n}, |J|= \delta \biggr{\}}.
\end{equation}
We stress that $A^{*}(\delta)>0$. Indeed, if $J= \mathopen{[}\omega,\omega+\delta\mathclose{]}$ (for some $\omega\in\mathbb{R}$) with $J \subseteq J_{n} = \mathopen{[}\sigma_{n},\tau_{n}\mathclose{]}$ for some $n\in\{1,\ldots,N\}$, then the continuous map
\begin{equation*}
\mathopen{[}\sigma_{n}, \tau_{n}-\delta\mathclose{]} \ni \omega \mapsto A_{n}(\omega+\delta)-A_{n}(\omega) = \int_{J} a(t)\,\mathrm{d}t >0
\end{equation*}
achieves a positive minimum $a^{*}_{n}(\delta)$, where $A_{n}(\xi):=\int_{\sigma_{n}}^{\xi} a(t)\,\mathrm{d}t$, $\xi\in J_{n}$, is the primitive of the weight $a$ on $J_{n}$. Then, we have $A^{*}(\delta) = \min_{n} a^{*}_{n}(\delta)>0$.

\medskip

We say that the couple $(u,v)$ is a \textit{$T$-periodic solution} of \eqref{eq:system-S} if $u$ and $v$ are absolutely continuous functions, $T$-periodic, and solve the two differential equations in \eqref{eq:system-S} almost everywhere in $\mathbb{R}$. Moreover, we notice that if $(u,v)$ is a solution of \eqref{eq:system-S}, then $u\in\mathcal{C}^{1}(\mathbb{R})$, due to the continuity of the function $h$.

\begin{remark}\label{rem-2.1}
We notice that condition \eqref{eq:cond-a-diesis} is a natural assumption, since if we deal with a continuously differentiable increasing function $g$, then \eqref{eq:cond-a-diesis} is necessary for the existence of solutions $(u,v)$ of \eqref{eq:system-S} with $u>0$. Indeed, an integration leads to
\begin{align*}
\lambda \int_{0}^{T} a(t)\, \mathrm{d}t &= -\int_{0}^{T} \dfrac{v'(t)}{g(u(t))}\, \mathrm{d}t 
= \biggl{[} \dfrac{v(t)}{g(u(t))}\biggr{]}_{t=0}^{t=T} - \int_{0}^{T} \dfrac{v(t)g'(u(t))u'(t)}{(g(u(t)))^{2}}\, \mathrm{d}t 
\\
&= - \int_{0}^{T} \dfrac{h(t,v(t))v(t)g'(u(t))}{(g(u(t)))^{2}}\, \mathrm{d}t <0,
\end{align*}
where the last inequality follow by \eqref{cond-h-1} and the monotonicity of $a$.
\hfill$\lhd$
\end{remark}

\medskip

The main results of this paper read as follows.

\begin{theorem}\label{th-main}
Let $h\colon\mathbb{R}\times\mathbb{R}\to\mathbb{R}$ be a continuous function, $T$-periodic in the first variable, and satisfying \ref{cond-h-0} and \ref{cond-h-2}.
Let $a\colon\mathbb{R}\to\mathbb{R}$ be a locally integrable $T$-periodic function satisfying \eqref{eq:cond-a-diesis} and \ref{cond-a-star}. Let $g\colon\mathopen{[}0,+\infty\mathclose{[}\to\mathopen{[}0,+\infty\mathclose{[}$ be a continuous function, regularly oscillating at zero, satisfying \eqref{eq:cond-g-star} and 
\begin{equation}\label{main-cond-zero}
\lim_{s\to 0^{+}} \dfrac{h(t, K g(s))}{s} = 0,
\quad \text{uniformly in $t$, for every $K\in\mathbb{R}$.}
\end{equation}
Then, there exists $\lambda^{*}>0$ such that for every $\lambda>\lambda^{*}$, system \eqref{eq:system-S} has at least one $T$-periodic solution $(u,v)$ with $u(t)>0$ for all $t\in\mathbb{R}$.
\end{theorem}

\begin{theorem}\label{th-main-2}
Let $h\colon\mathbb{R}\times\mathbb{R}\to\mathbb{R}$ be a continuous function, $T$-periodic in the first variable, and satisfying \ref{cond-h-0} and \ref{cond-h-2}. 
Let $a\colon\mathbb{R}\to\mathbb{R}$ be a locally integrable $T$-periodic function satisfying \eqref{eq:cond-a-diesis} and \ref{cond-a-star}. Let $g\colon\mathopen{[}0,+\infty\mathclose{[}\to\mathopen{[}0,+\infty\mathclose{[}$ be a continuous function, regularly oscillating at zero, satisfying \eqref{eq:cond-g-star}, \eqref{main-cond-zero} and 
\begin{equation}\label{main-cond-infinity}
\lim_{s\to+\infty} \mathrm{sign}(K) \frac{h(t,Kg(s))}{s} = +\infty, 
\quad \text{uniformly in $t$, for every $K\in\mathbb{R}\setminus\{0\}$.}
\end{equation}
Then, for every $\lambda>0$, system \eqref{eq:system-S} has at least one $T$-periodic solution $(u,v)$ with $u(t)>0$ for all $t\in\mathbb{R}$.
\end{theorem}

\begin{remark}\label{rem-2.2}
Theorem~\ref{th-main} and Theorem~\ref{th-main-2} remain valid if we replace the hypothesis that $g$ is regularly oscillating with the hypothesis that $g$ is continuously differentiable and such that
\begin{equation*}
\limsup_{s\to 0^{+}} \dfrac{g'(s)s}{g(s)} \leq C_{g},
\quad \text{for some $C_{g}\geq0$.}
\end{equation*}
This will become clear in Section~\ref{section-4.1}.
\hfill$\lhd$
\end{remark}

\begin{remark}[Neumann boundary conditions]\label{rem-2.3}
We observe that the existence results stated in Theorem~\ref{th-main} and Theorem~\ref{th-main-2} are also valid when dealing with system \eqref{eq:system-S} associated with the boundary conditions $v(0)=v(T)=0$ (or equivalently $u'(0)=u'(T)=0$).
The analysis in the present paper can be adapted to this case following the same strategy by fixing the technicalities as in \cite[Chapter~3]{Fe-2018}. Concerning the related maximum principles we also refer to Remark~\ref{rem-A.1}.
\hfill$\lhd$
\end{remark}

\subsection{Maximum principles}\label{section-2.1}

The proofs of Theorem~\ref{th-main} and its variants are based on the topological degree theory and accordingly we will introduce operators defined in Banach spaces. With this aim, we extend the nonlinear functions in \eqref{eq:system-S} on the whole real line. Accordingly, we define $f\colon\mathbb{R}\times\mathbb{R}\to\mathbb{R}$
\begin{equation*}
f(t,s) :=
\begin{cases}
\, \lambda a(t) g(s), &\text{if $s\in\mathopen{[}0,+\infty\mathclose{[}$,}
\\
\, -s, &\text{if $s\in\mathopen{]}-\infty,0\mathclose{[}$,}
\end{cases}
\end{equation*}
and consider the planar system
\begin{equation}\leqnomode
\label{eq:system-S-tilde}
\begin{cases}
\, u' = h(t,v),
\\
\, v' = -f(t,u).
\end{cases}
\tag{$\tilde{\mathcal{S}}$}
\end{equation}
By the weak maximum principle (cf.~Proposition~\ref{weak-max-p}), since $-f(t,s)<0$ for a.e.~$t\in\mathbb{R}$ and for all $s\in\mathopen{]}-\infty,0\mathclose{[}$, we have that a $T$-periodic solution $(u,v)$ of \eqref{eq:system-S-tilde} is such that $u(t)\geq 0$ for all $t\in\mathbb{R}$ and thus $(u,v)$ is also a $T$-periodic solution of \eqref{eq:system-S}.

Now, we present a strong maximum principle which ensures that the first component of a non-trivial $T$-periodic solution of \eqref{eq:system-S-tilde} is a positive function. 

\begin{proposition}\label{strong-max-p}
Let $h\colon\mathbb{R}\times\mathbb{R}\to\mathbb{R}$ be a continuous function, $T$-periodic in the first variable, and satisfying \ref{cond-h-0} and \ref{cond-h-2}. Let $a\colon\mathbb{R}\to\mathbb{R}$ be a locally integrable $T$-periodic function satisfying \ref{cond-a-star}. Let $g\colon\mathopen{[}0,+\infty\mathclose{[}\to\mathopen{[}0,+\infty\mathclose{[}$ be a continuous function satisfying \eqref{eq:cond-g-star}.
Assume that for every constant $K>0$ there exists $\varepsilon>0$ and $\beta>0$ such that
\begin{equation}\label{cond-strong-max-p}
h(t, K g(s)) \leq \beta s, \quad \text{for all $t\in\mathbb{R}$ and $s\in\mathopen{[}0,\varepsilon\mathclose{[}$.}
\end{equation}
If $(u,v)$ is a non-trivial $T$-periodic solution of \eqref{eq:system-S-tilde}, then $u(t)>0$ for all $t\in\mathbb{R}$.
\end{proposition}

\begin{proof}
Let $(u,v)$ be a non-trivial $T$-periodic solution of \eqref{eq:system-S-tilde}. By Proposition~\ref{weak-max-p}, we have that $u(t)\geq 0$ for all $t\in\mathbb{R}$ and thus $(u,v)$ solves \eqref{eq:system-S}. By contradiction, we suppose that there exists $t^{*}\in\mathbb{R}$ such that $u(t^{*})=0$. We notice that $u'(t^{*})=0$ and consequently $h(t^{*},v(t^{*}))=0$ (by exploiting the first equation in \eqref{eq:system-S-tilde}) and $v(t^{*})=0$ (by \ref{cond-h-0} and \ref{cond-h-2}).
Since $u\not\equiv0$ (otherwise $(u,v)\equiv(0,0)$), we have that there exists $t_{1}\in\mathbb{R}$ with $t_{1}>t^{*}$ and $u(t_{1})>0$. We define $t_{0}:=\max\{t\in\mathopen{[}t^{*},t_{1}\mathclose{]}\colon u(t)=0\}$. Therefore, $u(t_{0})=u'(t_{0})=v(t_{0})=0$ and $u(t)>0$ for all $t\in\mathopen{]}t_{0},t_{1}\mathclose{]}$. Thus, we have
\begin{equation*}
v(t) = v(t_{0}) - \lambda \int_{t_{0}}^{t} a(\xi) g(u(\xi)) \,\mathrm{d}\xi 
\leq \lambda \int_{t_{0}}^{t} a^{-}(\xi) g(u(\xi)) \,\mathrm{d}\xi,
\quad \text{for all $t\in\mathopen{[}t_{0},t_{1}\mathclose{]}$,}
\end{equation*}
and, by \ref{cond-h-2},
\begin{align*}
u(t) &= u(t_{0}) + \int_{t_{0}}^{t} h(\xi,v(\xi)) \,\mathrm{d}\xi 
\\
&\leq \int_{t_{0}}^{t} h \biggl{(}\xi,\lambda \int_{t_{0}}^{\xi} a^{-}(s) g(u(s)) \,\mathrm{d}s \biggr{)} \,\mathrm{d}\xi,
\quad \text{for all $t\in\mathopen{[}t_{0},t_{1}\mathclose{]}$.}
\end{align*}
Choose $K:= \lambda \|a^{-}\|_{L^{1}(t_{0},t_{1})}$ and consider $\varepsilon, \beta>0$ satisfying inequality \eqref{cond-strong-max-p}. Let $\delta>0$ be such that $t_{0}+\delta<t_{1}$, $\delta \beta < 1$, and $u(t)<\varepsilon$ for each $t\in\mathopen{[}t_{0},t_{0}+\delta\mathclose{]}$.
Let $\hat{t}$ be a maximum point of $u$ in $\mathopen{[}t_{0},t_{0}+\delta\mathclose{]}$ and $\bar{t}$ a maximum point of $g\circ u$ in $\mathopen{[}t_{0},t_{0}+\delta\mathclose{]}$.
Then, we deduce
\begin{align*}
u(\hat{t}) 
&\leq \int_{t_{0}}^{t_{0}+\delta} h \biggl{(}\xi,\lambda \int_{t_{0}}^{\xi} a^{-}(s) g(u(s)) \,\mathrm{d}s \biggr{)} \,\mathrm{d}\xi
\\
&\leq \int_{t_{0}}^{t_{0}+\delta} h (\xi,\lambda \|a^{-}\|_{L^{1}(t_{0},t_{1})} g(u(\bar{t})) ) \,\mathrm{d}\xi
\leq \int_{t_{0}}^{t_{0}+\delta} \beta u(\bar{t}) \,\mathrm{d}\xi \leq \delta \beta u(\hat{t}) < u(\hat{t}),
\end{align*}
a contradiction.
The proof is complete.
\end{proof}

When $a$ is essentially bounded (at least in the intervals where it is negative), a different version of the above strong maximum principle can be stated by exploiting Proposition~\ref{strong-max-time-maps}. 
Preliminarily, let us introduce the primitives $\underline{H}$ and $G$ of $\underline{h}$ and $g$, respectively, that is
\begin{equation*}
\underline{H}(s):=\int_{0}^{s}\underline{h}(\xi) \,\mathrm{d}\xi, 
\quad
G(s):=\int_{0}^{s} g(\xi) \,\mathrm{d}\xi, 
\end{equation*}
and denote by $\underline{H}^{-1}_{\mathrm{l}}$ and $\underline{H}^{-1}_{\mathrm{r}}$ the left and right inverse of $\underline{H}$, respectively.

\begin{proposition}\label{strong-max-p-2}
Let $h\colon\mathbb{R}\times\mathbb{R}\to\mathbb{R}$ be a continuous function, $T$-periodic in the first variable, and satisfying \ref{cond-h-0} and \ref{cond-h-2}. Let $a\in L^{1}_{\mathrm{loc}}(\mathbb{R})\cap L^{\infty}(\mathbb{R})$ be $T$-periodic function satisfying \ref{cond-a-star}. Let $g\colon\mathopen{[}0,+\infty\mathclose{[}\to\mathopen{[}0,+\infty\mathclose{[}$ be a continuous function satisfying \eqref{eq:cond-g-star}.
If at least one of the integrals
\begin{equation}\label{integr-strong-max-p-2}
\int_{0}^{\varepsilon} \dfrac{\mathrm{d}s}{\overline{h}(\underline{H}^{-1}_{\mathrm{l}}(\lambda \|a^{-}\|_{\infty}G(s)))}, 
\quad
\int_{0}^{\varepsilon} \dfrac{\mathrm{d}s}{\overline{h}(\underline{H}^{-1}_{\mathrm{r}}(\lambda \|a^{-}\|_{\infty}G(s)))}
\end{equation}
diverges (for every $\varepsilon>0$ sufficiently small), then every non-trivial $T$-periodic solution $(u,v)$ of \eqref{eq:system-S-tilde} is such that $u(t)>0$ for all $t\in\mathbb{R}$.
\end{proposition}

\begin{proof}
Let $(u,v)$ be a non-trivial $T$-periodic solution of \eqref{eq:system-S-tilde}. By Proposition~\ref{weak-max-p}, we have that $u(t)\geq 0$ for all $t\in\mathbb{R}$ and thus $(u,v)$ solves \eqref{eq:system-S}. By contradiction, we suppose that there exists $t_{0}\in\mathbb{R}$ such that $u(t_{0})=0$. 
Without loss of generality we can suppose that $t_{0}$ is such that $u(t)>0$ in $\mathopen{]}t_{0},t_{0}+\delta\mathclose{]}$ or in $\mathopen{[}t_{0}-\delta,t_{0}\mathclose{[}$, for some $\delta>0$. Let us consider the first alternative (the other is analogous reasoning backwards). Recalling hypothesis \ref{cond-a-star}, we prove that we can take $\delta>0$ sufficiently small such that $a\leq0$ a.e.~in $\mathopen{[}t_{0},t_{0}+\delta\mathclose{]}$. Indeed, if $a>0$ a.e.~in $\mathopen{[}t_{0},t_{0}+\delta\mathclose{]}$ (for all $\delta>0$ small), then $v(t) = -\lambda \int_{t_{0}}^{t} a(\xi) g(u(\xi))\,\mathrm{d}\xi < 0$, for all $t\in\mathopen{]}t_{0},t_{0}+\delta\mathclose{]}$,
and so $u'(t)<0$ (by \eqref{cond-h-1} and the first equation in \eqref{eq:system-S-tilde}) and $u(t)<0$ for all $t\in\mathopen{]}t_{0},t_{0}+\delta\mathclose{]}$ (since $u(t_{0})=0$), a contradiction. Notice also that $a\not\equiv0$ in $\mathopen{[}t_{0},t_{0}+\delta\mathclose{]}$, otherwise $u\equiv0$ in $\mathopen{[}t_{0},t_{0}+\delta\mathclose{]}$ (since $u(t_{0})=0$), contradicting the assumptions.

We conclude that $0 \leq -\lambda a(t) g(u(t)) = \lambda a^{-}(t) g(u(t)) \leq \lambda \|a^{-}\|_{\infty} G(u(t))$ for a.e.~$t\in\mathopen{]}t_{0},t_{0}+\delta\mathclose{]}$. We can now apply Proposition~\ref{strong-max-time-maps}, with $k(t,s)=-\lambda a(t) g(s)$ and $\overline{k}(s)=\lambda \|a^{-}\|_{\infty}G(s)$, to complete the proof.
\end{proof}

In Appendix~\ref{appendix-A} we give a more general version of the maximum principles which may have an independent interest. The simplified form given in Proposition~\ref{strong-max-p} is however enough for our applications in the next sections.

\section{Abstract setting and strategy of the proof}\label{section-3}

We start by defining the suitable spaces and operators for the application of the coincidence degree theory by Mawhin (we refer the reader to \cite{GaMa-1977,Ma-1979,Ma-1993} for more details).

Given $T>0$, let $X:= \mathcal{C}_{T}$ be the Banach space of continuous $T$-periodic functions $(u,v)\colon\mathbb{R}\to \mathbb{R}^{2}$, endowed with the norm
\begin{equation*}
\|(u,v)\|_{\infty}:=\max \bigl{\{}\|u\|_{\infty}, \|v\|_{\infty} \bigr{\}},
\qquad \text{where $\|w\|_{\infty}:=\max_{t\in\mathbb{R}} |w(t)|$,}
\end{equation*}
and let $Z:= L^{1}_{T}$ be the Banach space of locally integrable and $T$-periodic functions $(x,y)\colon\mathbb{R}\to\mathbb{R}^{2}$ with the $L^{1}$-norm 
\begin{equation*}
\|(x,y)\|_{L^{1}_{T}}:=\max \biggl{\{}\int_{0}^{T} |x(t)|\,\mathrm{d}t, \int_{0}^{T} |y(t)|\,\mathrm{d}t \biggr{\}}.
\end{equation*}
On the set
\begin{equation*}
\mathrm{dom}\,L := \Bigl{\{} (u,v)\in X \colon (u,v) \text{ is absolutely continuous} \Bigr{\}}
\end{equation*}
we define the linear operator $L\colon \mathrm{dom}\,L \to Z$ as
\begin{equation*}
[L(u,v)](t) := (u'(t),v'(t)),\quad t\in \mathbb{R}.
\end{equation*}
We observe that $\ker L \equiv {\mathbb{R}^{2}}$ is made up of the constant functions $(u,v)$ and
\begin{equation*}
\mathrm{Im}\,L = \biggl\{(x,y)\in Z \colon \int_{0}^{T} x(t) \, \mathrm{d}t=0, \  \int_{0}^{T} y(t) \, \mathrm{d}t=0\biggr\}.
\end{equation*}
Next, we introduce the projections $P\colon X \to \ker L$ and $Q\colon Z \to \mathrm{coker}\,L$ as
\begin{equation*}
P, Q \colon (u,v) \mapsto \left(\dfrac{1}{T}\int_{0}^{T}u(t) \, \mathrm{d}t,\dfrac{1}{T}\int_{0}^{T}v(t) \, \mathrm{d}t\right).
\end{equation*}
Hence, $\mathrm{coker}\,L \equiv {\mathbb{R}^{2}}$ and $\ker P$
is made up of the periodic functions $(u,v)$ with mean value zero.
Moreover, the right inverse linear operator $K_{P}$ is the map which,
to every $(x,y)\in \mathrm{Im}\, L$, associates the unique solution $(u,v)$ of
\begin{equation*}
\begin{cases}
\, u'(t)=x(t),
\\
\,v'(t)=y(t),
\end{cases}
\quad 
\begin{cases}
\, u(0)=u(T), 
\\
\, v(0)=v(T),
\end{cases}
\quad \int_{0}^{T} u(t) \, \mathrm{d}t = 0, \; \int_{0}^{T} v(t) \, \mathrm{d}t = 0.
\end{equation*}
At last, we fix the identity vector field in ${\mathbb{R}^{2}}$ as a linear isomorphism $J \colon \mathrm{coker}\,L \to \ker L$. 

Let us denote by $N \colon X \to Z$ the Nemytskii operator induced by $h$ and $-f$, that is
\begin{equation*}
[N(u,v)](t):= (h(t,v(t)),-f(t,u(t))), \quad t\in \mathbb{R}.
\end{equation*}

All the structural assumptions required by Mawhin's theory (that is $L$ is Fredholm of index zero and $N$ is $L$-completely continuous) are satisfied by standard facts. As a consequence, all the solution of \eqref{eq:system-S-tilde} are solutions of the coincidence equation
\begin{equation}\label{eq-2.11}
L(u,v) = N(u,v), \quad (u,v)\in \mathrm{dom}\,L,
\end{equation}
and viceversa; moreover \eqref{eq-2.11} is equivalent to the fixed point problem
\begin{equation}\label{eq-2.3}
(u,v) = \Phi(u,v):= P(u,v) + JQN(u,v) + K_{P}(\mathrm{Id}-Q)N(u,v), \quad (u,v)\in X,
\end{equation}
and we can apply the Leray--Schauder degree theory to the operator equation \eqref{eq-2.3}. We recall that, for an open and bounded set $\Omega\subseteq X$ such that
\begin{equation*}
L(u,v) \neq N(u,v), \quad \text{for all $(u,v)\in \mathrm{dom}\,L \cap \partial\Omega$,}
\end{equation*}
the \textit{coincidence degree of $L$ and $N$ in $\Omega$} is defined as
\begin{equation*}
\mathrm{D}_{L}(L-N,\Omega):= \mathrm{deg}_{\mathrm{LS}}(\mathrm{Id}-\Phi,\Omega,0),
\end{equation*}
where ``$\mathrm{deg}_{\mathrm{LS}}$'' denotes the Leray--Schauder degree; in the sequel we also denote by ``$\mathrm{deg}_{\mathrm{B}}$'' the Brouwer degree. Notice that $\mathrm{D}_{L}$ is independent on the choice of $P$ and $Q$, and also of $J$, provided that we have fixed an orientation on $\ker L$ and $\mathrm{coker}\,L$ and considered for $J$ only orientation-preserving isomorphisms. The coincidence degree has all the standard properties of the Leray--Schauder degree, in particular, if it holds that $\mathrm{D}_{L}(L-N,\Omega)\neq0$, then \eqref{eq-2.11} has at least one solution in $\Omega$.

In the sequel, we will make use of the following two results regarding the computation of the degree via the homotopy invariance property of $\mathrm{D}_{L}$ (cf.~\cite{FeZa-2015,Ma-1969,Ma-1972}).

\begin{lemma}[Mawhin, 1972]\label{lemma_Mawhin}
Let $L$ and $N$ be as above and let $\Omega\subseteq X$ be an open and bounded set.
Suppose that
\begin{equation*}
L(u,v) \neq \vartheta N(u,v), \quad \text{for all $(u,v)\in \mathrm{dom}\,L \cap \partial\Omega$ and $\vartheta\in\mathopen{]}0,1\mathclose{]}$,}
\end{equation*}
and
\begin{equation*}
QN(u,v)\neq0, \quad \text{for all $(u,v)\in \partial\Omega \cap \ker L$.}
\end{equation*}
Then
\begin{equation*}
\mathrm{D}_{L}(L-N,\Omega) = \mathrm{deg}_{\mathrm{B}}(-JQN|_{\ker L},\Omega \cap \ker L,0).
\end{equation*}
\end{lemma}

\begin{lemma}\label{lemma-2.2}
Let $L$ and $N$ be as above and let $\Omega\subseteq X$ be an open and bounded set.
Suppose that $(w_{1},w_{2})\neq0$ is a vector such that
\begin{equation*}
L(u,v) \neq N(u,v) + \alpha (w_{1},w_{2}), \quad \text{for all $(u,v)\in \mathrm{dom}\,L \cap \partial\Omega$ and $\alpha \geq 0$.}
\end{equation*}
Then
\begin{equation*}
\mathrm{D}_{L}(L-N,\Omega) = 0.
\end{equation*}
\end{lemma}

\medskip

Given this abstract setting, we can now present the strategy of the proof of Theorem~\ref{th-main} and Theorem~\ref{th-main-2}.
Let us first consider the parameter-dependent coincidence equation
\begin{equation*}
L(u,v) = \vartheta N(u,v), \quad (u,v)\in \mathrm{dom}\,L, \quad \vartheta\in\mathopen{[}0,1\mathclose{]},
\end{equation*}
which is equivalent to the $T$-periodic problem associated with
\begin{equation}\label{sys-tilde-S-theta}
\begin{cases}
\, u' = \vartheta h(t,v),
\\
\, v' = -\vartheta f(t,u).
\end{cases}
\end{equation}
Let $(u,v)$ be a $T$-periodic solution of \eqref{sys-tilde-S-theta} for some $\vartheta \in \mathopen{[}0,1\mathclose{]}$.  By the weak maximum principle (cf.~Proposition~\ref{weak-max-p}), since $-f(t,s)<0$ for a.e.~$t\in\mathbb{R}$ and for all $s\in\mathopen{]}-\infty,0\mathclose{[}$, we have that $u(t)\geq 0$ for every $t\in \mathbb{R}$. Moreover, by hypothesis \eqref{main-cond-zero}, from the strong maximum principle (cf.~Proposition~\ref{strong-max-p}), it follows that $u(t) > 0$ for all $t\in\mathbb{R}$. Therefore, we can focus our attention on the $T$-periodic solutions $(u,v)$ of
\begin{equation}\label{sys-tilde-S-theta}
\begin{cases}
\, u' = \vartheta h(t,v),
\\
\, v' = -\vartheta a(t)g(u),
\end{cases}
\end{equation}
with $u(t)>0$ for all $t\in\mathbb{R}$.

In Section~\ref{section-4.1}, we show that the hypothesis \eqref{main-cond-zero} of Theorem~\ref{th-main} (and of Theorem~\ref{th-main-2}) ensures the following crucial property for the application of Lemma~\ref{lemma_Mawhin}.
\begin{enumerate}[leftmargin=35pt,labelsep=10pt,label=\textup{$(\mathscr{H}_{r_{0}})$}]
\item There exists $r_{0}>0$ such that for all $r\in\mathopen{]}0,r_{0}\mathclose{]}$ and for all $\vartheta \in\mathopen{]}0,1\mathclose{]}$ there is no $T$-periodic solution $(u,v)$ of \eqref{sys-tilde-S-theta} such that $u(t)>0$ for all $t\in\mathbb{R}$ and $\|u\|_{\infty}=r$.
\label{cond-H-r0}
\end{enumerate}
As a consequence, for the open and bounded set
\begin{equation*}
\Omega_{r_{0}}:= \bigl{\{} (u,v)\in X \colon \|(u,v)\|_{\infty} < r_{0} \bigr{\}},
\end{equation*}
it holds that
\begin{equation*}
L(u,v) \neq \vartheta N(u,v),
\quad \text{for all $(u,v)\in \mathrm{dom}\,L \cap \partial\Omega_{r_{0}}$ and $\vartheta\in \mathopen{]}0,1\mathclose{]}$.}
\end{equation*}

Consider now $(u,v)\in \partial\Omega_{r_{0}} \cap \ker L$. In this case, $(u,v)\equiv (U,V)\in {\mathbb{R}^{2}}$, with $|U| = r_{0}$ or $|V|=r_{0}$, and
\begin{equation*}
-JQN(U,V) = \left(-\dfrac{1}{T} \int_{0}^{T} h(t,V)\, \mathrm{d}t, \dfrac{1}{T} \int_{0}^{T} f(t,U)\, \mathrm{d}t\right) =: (-h^{\#}(V),f^{\#}(U)).
\end{equation*}
Notice also that $\Omega_{r_{0}}\cap \ker L = \mathopen{]}-r_{0},r_{0}\mathclose{[}\times\mathopen{]}-r_{0},r_{0}\mathclose{[}$.
By the definition of $f$, we have that
\begin{equation*}
f^{\#}(s)= \dfrac{1}{T} \int_{0}^{T} f(t,s) \, \mathrm{d}t =
\begin{cases}
\, - \dfrac{g(s)}{T} \displaystyle\int_{0}^{T} a(t) \, \mathrm{d}t, & \text{if $s >  0$;} 
\\
\, s, & \text{if $s \leq 0$.}
\end{cases}
\end{equation*}
Therefore, by hypotheses \ref{cond-h-2}, \eqref{eq:cond-g-star} and \eqref{eq:cond-a-diesis}, it follows that $QNu\neq 0$ for each $u\in \partial\Omega_{r_{0}} \cap \ker L$ and, moreover,
\begin{equation*}
\mathrm{deg}_{\mathrm{B}} ((-h^{\#},f^{\#}), \mathopen{]}-r_{0},r_{0}\mathclose{[}\times \mathopen{]}-r_{0},r_{0}\mathclose{[},0) =-1,
\end{equation*}
by a careful study of the signs of $(-h^{\#},f^{\#})$ in each of the four quadrants.
By Lemma~\ref{lemma_Mawhin} we conclude that
\begin{equation}\label{eq-deg1}
\mathrm{D}_{L}(L-N,\Omega_{r_{0}}) = -1.
\end{equation}

Secondly we study the parameter-dependent operator equation
\begin{equation}\label{eq-2.14}
L(u,v) = N(u,v) - \alpha (0,w), \quad (u,v)\in \mathrm{dom}\,L,
\end{equation}
for $\alpha \geq 0$ and a non-negative locally integrable $T$-periodic function $w$ with $w\equiv 0$ in $(\mathbb{R}/T\mathbb{Z}) \setminus \bigcup_{n=1}^{N}J_{n}$. Equation \eqref{eq-2.14} is equivalent to the $T$-periodic problem associated with
\begin{equation}\label{sys-tilde-S-alpha}
\begin{cases}
\, u'=h(t,v), \\
\, v'=-\lambda f(t,u(t))-\alpha w(t).
\end{cases}
\end{equation}
Let $(u,v)$ be a $T$-periodic solution of \eqref{sys-tilde-S-alpha} for some $\alpha \in \mathopen{[}0,+\infty\mathclose{[}$. Observing that $-f(t,s)-\alpha w(t) = s -\alpha w(t) < 0$ for a.e.~$t\in\mathbb{R}$ and for all $s\in\mathopen{]}-\infty,0\mathclose{[}$, by the weak maximum principle, we deduce that $u(t)\geq0$ for all $t\in\mathbb{R}$. Hence, we can deal with the $T$-periodic solutions $(u,v)$ of
\begin{equation}\label{syst-alpha-2}
\begin{cases}
\, u'=h(t,v), \\
\, v'=-\lambda a(t) g(u(t)) -\alpha w(t),
\end{cases}
\end{equation}
with $u(t)\geq 0$ for all $t\in\mathbb{R}$.
We stress that in this case we do not exploit the strong maximum principle, but we just need the weak version.

The main results of Section~\ref{section-4.2} assure the application of Lemma~\ref{lemma-2.2}. 
Indeed, in Section~\ref{section-4.2} we prove that under the hypotheses of Theorem~\ref{th-main} we can choose $R>r_{0}$ such that:
\begin{enumerate}[leftmargin=39pt,labelsep=10pt,label=\textup{$(\mathscr{H}_{R}^{\lambda^{*}})$}]
\item There exists $\lambda^{*}=\lambda^{*}(R)$ such that, for every $\alpha\geq 0$, for every non-negative locally integrable $T$-periodic function $w$ with $w\equiv 0$ in $(\mathbb{R}/T\mathbb{Z})\setminus\bigcup_{n=1}^{N}J_{n}$, for every $\lambda>\lambda^{*}$, there are no $T$-periodic solutions $(u,v)$ of system \eqref{syst-alpha-2} with $0\leq u(t) \leq \|u\|_{\infty}=R$ for all $t\in\mathbb{R}$;
\label{cond-H-lambda}
\end{enumerate}
while under the hypotheses of Theorem~\ref{th-main-2} we have the following:
\begin{enumerate}[leftmargin=39pt,labelsep=10pt,label=\textup{$(\mathscr{H}^{\lambda})$}]
\item For every $\lambda>0$, there exists $R>0$ such that, for every $\alpha\geq 0$, for every non-negative locally integrable $T$-periodic function $w$ with $w\equiv 0$ in $(\mathbb{R}/T\mathbb{Z})\setminus\bigcup_{n=1}^{N}J_{n}$, there are no $T$-periodic solutions $(u,v)$ of system \eqref{syst-alpha-2} with $0\leq u(t) \leq \|u\|_{\infty}=R$ for all $t\in\mathbb{R}$.
\label{cond-H-lambda-00}
\end{enumerate}
For now on, we fix $R$ and $\lambda$ as above.
An integration in $\mathopen{[}0,T\mathclose{]}$ of the second equation in \eqref{syst-alpha-2} leads to
\begin{equation*}
\lambda \int_{0}^{T} a(t) g(u(t))\,\mathrm{d}t = - \alpha \int_{0}^{T} w(t) \, \mathrm{d}t = - \alpha \|w\|_{L^{1}}.
\end{equation*}
Therefore, we deduce that there are no $T$-periodic solutions $(u,v)$ of \eqref{syst-alpha-2} with $\|u\|_{\infty}\leq R$ if $\alpha\geq \alpha_{0}$, where
\begin{equation*}
\alpha_{0} > \dfrac{\lambda \|a\|_{L^{1}} \overline{g}(R)}{\|w\|_{L^{1}}}.
\end{equation*}
Moreover, for $\alpha\in\mathopen{[}0,\alpha_{0}\mathclose{]}$, let $(u,v)$ be a $T$-periodic solutions of \eqref{syst-alpha-2} with $\|u\|_{\infty}\leq R$ and denote by $t^{*}\in\mathbb{R}$ a maximum point of $u$. Clearly $u'(t^{*})=0$ and thus $v(t^{*})=0$, by hypotheses \ref{cond-h-0} and \ref{cond-h-2}.
Integrating again the second equation in \eqref{syst-alpha-2}, we have
\begin{align*}
|v(t)|
&= \biggl{|} v(t^{*})+\lambda \int_{t^{*}}^{t} a(\xi) g(u(\xi))\,\mathrm{d}\xi + \alpha \int_{t^{*}}^{t} w(\xi) \, \mathrm{d}\xi\biggr{|}
\\
&\leq \lambda \|a\|_{L^{1}} \overline{g}(R) + \alpha_{0} \|w\|_{L^{1}},
\end{align*}
for all $t\in\mathbb{R}$.
Let
\begin{equation*}
R' > \max\bigl{\{} r_{0}, \lambda \|a\|_{L^{1}} \overline{g}(R) + \alpha_{0} \|w\|_{L^{1}}\bigr{\}}.
\end{equation*}
Hence, for the open and bounded set
\begin{equation*}
\Omega_{R,R'}:= \bigl{\{} (u,v)\in X \colon \|u\|_{\infty}<R, \|v\|_{\infty} < R' \bigr{\}},
\end{equation*}
it holds that
\begin{equation*}
L(u,v) \neq  N(u,v) -\alpha (0,w),
\quad \text{for all $(u,v)\in \mathrm{dom}\,L \cap \partial\Omega_{R,R'}$ and $\alpha\geq 0$.}
\end{equation*}
According to Lemma~\ref{lemma-2.2} we have that
\begin{equation}\label{eq-deg0}
\mathrm{D}_{L}(L-N,\Omega_{R,R'}) = 0.
\end{equation}
In conclusion, from \eqref{eq-deg1}, \eqref{eq-deg0}, and the additivity property of the coincidence degree, we find that
\begin{equation*}
\mathrm{D}_{L}(L-N, \Omega_{R,R'} \setminus \mathrm{cl}(\Omega_{r_{0}})) = 1.
\end{equation*}
This ensures the existence of a solution $(\tilde{u},\tilde{v})$ to \eqref{eq-2.11} with $(\tilde{u},\tilde{v})\in \Omega_{R,R'} \setminus \mathrm{cl}(\Omega_{r_{0}})$. Hence, $(\tilde{u},\tilde{v})$ is a non-trivial solution of system \eqref{eq:system-S}.
Since $(\tilde{u},\tilde{v})\not\equiv(0,0)$, by the strong maximum principle we have that $\tilde{u}(t)>0$ for all $t\in \mathbb{R}$. The proofs of Theorem~\ref{th-main} and of Theorem~\ref{th-main-2} are complete.
\hfill\qed

\section{Qualitative results}\label{section-4}

In this section, we present some qualitative results concerning ``small'' and ``large'' solutions to \eqref{sys-tilde-S-theta} and \eqref{syst-alpha-2}, respectively. As illustrated in Section~\ref{section-3}, these results are essential for the application of the coincidence degree theory to obtain the existence of periodic solutions to \eqref{eq:system-S}. 

\subsection{Small solutions}\label{section-4.1}

The following result shows that the ``superlinear condition'' \eqref{main-cond-zero} of Theorem~\ref{th-main} guarantees the non-existence of ``small'' $T$-periodic solutions $(u,v)$ of the planar system \eqref{eq:system-S} with $u(t)>0$ for all $t\in\mathbb{R}$, more precisely the validity of \ref{cond-H-r0}. Notice that condition \ref{cond-a-star} is not needed.

\begin{lemma}\label{lem-small-sol}
Let $h\colon\mathbb{R}\times\mathbb{R}\to\mathbb{R}$ be a continuous function, $T$-periodic in the first variable, and satisfying \ref{cond-h-0} and \ref{cond-h-2}.
Let $a\colon\mathbb{R}\to\mathbb{R}$ be a locally integrable $T$-periodic function satisfying \eqref{eq:cond-a-diesis}. Let $g\colon\mathopen{[}0,+\infty\mathclose{[}\to\mathopen{[}0,+\infty\mathclose{[}$ be a continuous function, regularly oscillating at zero, satisfying \eqref{eq:cond-g-star} and
\begin{equation}\label{eq-cond-lem-small-sol}
\lim_{s\to 0^{+}} \dfrac{h(t,\pm \|a\|_{L^{1}} g(s))}{s} = 0, \quad \text{uniformly in $t$.}
\end{equation}
Then, \ref{cond-H-r0} holds.
\end{lemma}

\begin{proof}
The proof follows a similar scheme of the one of \cite[Theorem~3.1]{FeZa-2015}.
By contradiction, we suppose that for all $n\in\mathbb{N}\setminus\{0\}$ there exist $r_{n}\in\mathopen{]}0,\frac{1}{n}\mathclose{[}$, $\vartheta_{n}\in \mathopen{]}0,1\mathclose{]}$, and a $T$-periodic solution $(u_{n},v_{n})$ of \eqref{sys-tilde-S-theta} (with $\vartheta=\vartheta_{n}$) such that $u_{n}(t)>0$ for all $t\in\mathbb{R}$ and $\|u_{n}\|_{\infty}=r_{n}$.
We define
\begin{equation*}
w_{n}(t) := \dfrac{u_{n}(t)}{\|u_{n}\|_{\infty}}, \quad t\in\mathbb{R},
\end{equation*}
and claim that
\begin{equation}\label{claim-w}
\lim_{n\to+\infty} w_{n}(t) = 1,
\quad
\text{uniformly in $t$.}
\end{equation}
By Rolle's theorem, let $\hat{t}_{n}\in\mathbb{R}$ be such that $u_{n}'(\hat{t}_{n})=0$. Therefore, by \ref{cond-h-0}, \ref{cond-h-2}, and the first equation in \eqref{sys-tilde-S-theta}, $v_{n}(\hat{t}_{n})=0$. Then
\begin{equation*}
v_{n}(t) = v_{n}(\hat{t}_{n}) - \int_{\hat{t}_{n}}^{t} \vartheta_{n} a(\xi) g(u_{n}(\xi)) \,\mathrm{d}\xi 
= - \vartheta_{n} \int_{\hat{t}_{n}}^{t} a(\xi) g(u_{n}(\xi)) \,\mathrm{d}\xi, 
\quad \text{for all $t\in\mathbb{R}$,}
\end{equation*}
and so
\begin{equation*}
w_{n}'(t) = \dfrac{u_{n}'(t)}{\|u_{n}\|_{\infty}}
= \vartheta_{n} \dfrac{h(t, v_{n}(t))}{\|u_{n}\|_{\infty}}
= \vartheta_{n} \dfrac{h(t, -\vartheta_{n} \int_{\hat{t}_{n}}^{t} a(\xi) g(u_{n}(\xi)) \,\mathrm{d}\xi)}{\|u_{n}\|_{\infty}},
\quad \text{for all $t\in\mathbb{R}$.}
\end{equation*}
As a consequence, if $\bar{t}_{n}\in \mathbb{R}$ denotes the maximum of $g\circ u_{n}$, we obtain
\begin{align*}
\|w_{n}'\|_{\infty} 
&\leq \max_{t\in\mathbb{R}} \dfrac{|h(t, -\vartheta_{n} \int_{\hat{t}_{n}}^{t} a(\xi) g(u_{n}(\xi)) \,\mathrm{d}\xi)|}{\|u_{n}\|_{\infty}}
\\
&\leq \max_{t\in\mathbb{R}} \biggl{\{} \dfrac{h(t,\|a\|_{L^{1}}g(u_{n}(\bar{t}_{n})))}{r_{n}},-\dfrac{h(t,-\|a\|_{L^{1}}g(u_{n}(\bar{t}_{n})))}{r_{n}} \biggr{\}} \\
&\leq \max_{t\in\mathbb{R}} \biggl{\{} \dfrac{h(t,\|a\|_{L^{1}}g(u_{n}(\bar{t}_{n})))}{u_{n}(\bar{t}_{n})},-\dfrac{h(t,-\|a\|_{L^{1}}g(u_{n}(\bar{t}_{n})))}{u_{n}(\bar{t}_{n})} \biggr{\}},
\end{align*}
since $u_{n}(\bar{t}_{n}) \leq r_{n} = \|u_{n}\|_{\infty}$.
Therefore, by \eqref{eq-cond-lem-small-sol}, we have
\begin{equation*}
\lim_{n\to+\infty} w_{n}'(t) = 0,
\quad
\text{uniformly in $t$.}
\end{equation*}
Next, since $\|w_{n}\|_{\infty}=1$, there exists $\tilde{t}_{n}\in\mathbb{R}$ such that $w_{n}(\tilde{t}_{n})=1$ and thus
\begin{equation*}
w_{n}(t) = w_{n}(\tilde{t}_{n}) + \int_{\tilde{t}_{n}}^{t} w_{n}'(\xi)\,\mathrm{d}\xi 
= 1 + \int_{\tilde{t}_{n}}^{t} w_{n}'(\xi)\,\mathrm{d}\xi , 
\quad \text{for all $t\in\mathbb{R}$,}
\end{equation*}
and, since $w_{n}'(t)\to0$ uniformly as $n\to+\infty$, \eqref{claim-w} is proved.

Next, since $\vartheta_{n}>0$, we observe that
\begin{align*}
0 
&= -\dfrac{1}{\vartheta_{n}} \int_{0}^{T} v_{n}'(t)\,\mathrm{d}t
= \int_{0}^{T} a(t)g(u_{n}(t))\,\mathrm{d}t
\\
&=\int_{0}^{T} \bigl{[}a(t)g(r_{n}) + a(t) \bigl{(}g(r_{n} w_{n}(t)) - g(r_{n})\bigr{)}\bigr{]}\,\mathrm{d}t.
\end{align*}
Since $g(r_{n})>0$, from \eqref{eq:cond-a-diesis} we deduce that
\begin{align*}
0 &< -\int_{0}^{T} a(t) \,\mathrm{d}t
= \int_{0}^{T} a(t) \dfrac{g(r_{n} w_{n}(t)) - g(r_{n})}{g(r_{n})}\,\mathrm{d}t
\\
& \leq \|a\|_{L^{1}} \max_{t\in\mathbb{R}} \biggl{|} \dfrac{g(r_{n} w_{n}(t))}{g(r_{n})}-1\biggr{|} 
= \|a\|_{L^{1}} \biggl{|} \dfrac{g(r_{n} \hat{w}_{n})}{g(r_{n})}-1\biggr{|},
\end{align*}
where $\hat{w}_{n} := w_{n}(t_{n})$ for some $t_{n}\in\mathbb{R}$. A contradiction is obtained for $n$ sufficiently large, since $g$ is regularly oscillating at zero. 
The proof is complete.
\end{proof}

We can give an analogous result without assuming the regular oscillation at zero on $g$. In fact, this condition can be replaced just by asking $g$ to be continuously differentiable in a right neighborhood of zero and such that
\begin{equation}\label{hp-super}
\limsup_{s\to 0^{+}} \dfrac{g'(s)s}{g(s)} \leq C_{g},
\quad \text{for some $C_{g}\geq0$.}
\end{equation}
The precise result is the following.

\begin{lemma}\label{lem-small-sol-2}
Let $h\colon\mathbb{R}\times\mathbb{R}\to\mathbb{R}$ be a continuous function, $T$-periodic in the first variable, and satisfying \ref{cond-h-0} and \ref{cond-h-2}.
Let $a\colon\mathbb{R}\to\mathbb{R}$ be a continuous $T$-periodic function satisfying \eqref{eq:cond-a-diesis}. Let $g\colon\mathopen{[}0,+\infty\mathclose{[}\to\mathopen{[}0,+\infty\mathclose{[}$ be continuously differentiable in a right neighborhood of zero and satisfying \eqref{eq:cond-g-star}, \eqref{hp-super} and
\begin{equation}\label{cond-lem-small-sol-2}
\lim_{s\to 0^{+}} \dfrac{h(t, K g(s))}{s} = 0, \quad \text{uniformly in $t$, for every $K\in\mathbb{R}$.}
\end{equation}
Then, \ref{cond-H-r0} holds.
\end{lemma}

\begin{proof}
By contradiction, we suppose that for all $n\in\mathbb{N}\setminus\{0\}$ there exist $r_{n}\in\mathopen{]}0,\frac{1}{n}\mathclose{[}$, $\vartheta_{n}\in \mathopen{]}0,1\mathclose{]}$, and a $T$-periodic solution $(u_{n},v_{n})$ of \eqref{sys-tilde-S-theta} (with $\vartheta=\vartheta_{n}$) such that $u_{n}(t)>0$ for all $t\in\mathbb{R}$ and $\|u_{n}\|_{\infty}=r_{n}$.
Let $\bar{n}\in\mathbb{N}\setminus\{0\}$ be such that $g\in\mathcal{C}^{1}(\mathopen{]}0,1/\bar{n}\mathclose{[})$ and consider $n\geq \bar{n}$.
We define
\begin{equation*}
z_{n}(t) := \dfrac{v_{n}(t)}{\vartheta_{n} g(u_{n}(t))}, \quad t\in\mathbb{R},
\end{equation*}
then
\begin{align}
z_{n}'(t) 
&= \dfrac{v_{n}'(t)}{\vartheta_{n} g(u_{n}(t))} - \dfrac{v_{n}(t)}{\vartheta_{n} (g(u_{n}(t)))^{2}} g'(u_{n}(t))u_{n}'(t), 
\\
&= - a(t)-\dfrac{\vartheta_{n} h(t,v_{n}(t))v_{n}(t)}{\vartheta_{n} (g(u_{n}(t)))^{2}} g'(u_{n}(t))
\\
&= - a(t)- \vartheta_{n} \dfrac{h(t,\vartheta_{n} g(u_{n}(t)) z_{n}(t))}{g(u_{n}(t))} g'(u_{n}(t)) z_{n}(t),
\label{eq-zn}
\end{align}
for almost every $t\in\mathbb{R}$.
By Rolle's theorem, let $\hat{t}_{n}\in\mathbb{R}$ be such that $u_{n}'(\hat{t}_{n})=0$. Therefore, by \ref{cond-h-0}, \ref{cond-h-2}, and the first equation in \eqref{sys-tilde-S-theta}, $v_{n}(\hat{t}_{n})=0$. Then $z_{n}(\hat{t}_{n})=0$. 
Let $M$ and $\delta$ be such that
\begin{equation*}
M > \|a\|_{L^{1}}, \qquad 0 < \delta < \dfrac{M-\|a\|_{L^{1}}}{TM (C_{g}+1)}.
\end{equation*}
By \eqref{cond-lem-small-sol-2}, we fix $\varepsilon>0$ such that
\begin{equation*}
\dfrac{|h(t, \pm M g(s))|}{s} < \delta,
\quad \text{for all $t\in\mathbb{R}$ and $s\in\mathopen{]}0,\varepsilon\mathclose{]}$.}
\end{equation*}
and
\begin{equation*}
\dfrac{g'(s)s}{g(s)} \leq C_{g}+1, \quad \text{for all $s\in\mathopen{]}0,\varepsilon\mathclose{]}$.}
\end{equation*}
Let $n>1/\varepsilon$ and thus $u_{n}(t)\in\mathopen{]}0,\varepsilon\mathclose{[}$ for all $t\in\mathbb{R}$. We claim that
\begin{equation}\label{claim-zn}
\|z_{n}\|_{\infty} \leq M.
\end{equation}
By contradiction, if it is not true, let $I_{n}$ be the maximal interval of the form $\mathopen{[}\hat{t}_{n},\tau_{n}\mathclose{]}$ where $|z_{n}(t)| \leq M$. Then, $z_{n}(\hat{t}_{n})=0$ and $|z_{n}(\tau_{n})|=M$.
Integrating \eqref{eq-zn} on $I_{n}$ and passing to the absolute value, we have
\begin{align*}
M &= |z_{n}(\tau_{n})| \leq \|a\|_{L^{1}} + \vartheta_{n} \biggl{|} \int_{I_{n}} \dfrac{h(t,\vartheta_{n} g(u_{n}(t)) z_{n}(t))}{g(u_{n}(t))} g'(u_{n}(t)) z_{n}(t)\,\mathrm{d}t \biggr{|}
\\
&= \|a\|_{L^{1}} + \biggl{|} \int_{I_{n}} \dfrac{h(t,\vartheta_{n} g(u_{n}(t)) z_{n}(t))}{u_{n}(t)}  \dfrac{g'(u_{n}(t)) u_{n}(t)}{g(u_{n}(t))} z_{n}(t)\,\mathrm{d}t \biggr{|}
\\
&\leq \|a\|_{L^{1}} + \delta T M (C_{g}+1)  <M,
\end{align*}
a contradiction. The claim \eqref{claim-zn} is proved.

Next, integrating \eqref{eq-zn} on $\mathopen{[}0,T\mathclose{]}$, we obtain
\begin{align*}
0 &< - \int_{0}^{T} a(t) \, \mathrm{d}t = \vartheta_{n} \int_{0}^{T} \dfrac{h(t,\vartheta_{n} g(u_{n}(t)) z_{n}(t))}{u_{n}(t)}\dfrac{g'(u_{n}(t)) u_{n}(t)}{g(u_{n}(t))} z_{n}(t)\, \mathrm{d}t,
\\
&\leq M (C_{g}+1) \max_{s\in\mathopen{[}0,r_{n} \mathclose{]}} \int_{0}^{T} \dfrac{|h(t,\pm M g(s))|}{s}\,\mathrm{d}t,
\end{align*}
and, using \eqref{cond-lem-small-sol-2}, a contradiction is reached for $n\to+\infty$.
\end{proof}

\subsection{Large solutions}\label{section-4.2}

Preliminarily we show that the maximum of a non-trivial $T$-periodic solution of \eqref{eq:system-S} is reached in a positivity interval $J_{n}$.

\begin{lemma}\label{lem-maximum}
Let $h\colon\mathbb{R}\times\mathbb{R}\to\mathbb{R}$ be a continuous function, $T$-periodic in the first variable, and satisfying \ref{cond-h-0} and \ref{cond-h-2}. Let $a\colon\mathbb{R}\to\mathbb{R}$ be a locally integrable $T$-periodic function satisfying \ref{cond-a-star}. Let $g\colon\mathopen{[}0,+\infty\mathclose{[}\to\mathopen{[}0,+\infty\mathclose{[}$ be a continuous function satisfying \eqref{eq:cond-g-star}. Let $(u,v)$ be a non-trivial $T$-periodic solution of \eqref{eq:system-S} with $u(t)\geq0$ for each $t\in\mathbb{R}$. Then, there exist an index $n\in\{1,\ldots,N\}$ and a point $t^{*}\in J_{n}$ such that $u(t^{*})=\|u\|_{\infty}$. 
\end{lemma}

\begin{proof}
Let $t^{*}$ be a maximum point, i.e.~$u(t^{*})=\max_{t\in\mathbb{R}} u(t)$. Since $u(t^{*})>0$, there exists $\delta>0$ such that $u(t)>0$ in $\mathopen{[}t^{*},t^{*}+\delta\mathclose{]}$. Moreover, we observe that $u'(t^{*})=0$ and hence $v(t^{*})=0$. If $t^{*}\in\bigcup_{n=1}^{N} J_{n}$, the thesis is reached. Suppose that $t^{*}\notin \bigcup_{n=1}^{N} J_{n}$.
Let us start by assuming that
\begin{equation*}
a(t)\leq0, \quad \text{for a.e.~$t\in\mathopen{[}t^{*},t^{*}+\delta\mathclose{]}$,}
\qquad 
a \not\equiv 0, \quad \text{in~$\mathopen{[}t^{*},t^{*}+\delta\mathclose{]}$.}
\end{equation*}
Integrating the second equation in \eqref{eq:system-S} we have
\begin{equation}\label{eq-v-positive}
v(t)=v(t^{*})-\lambda \int_{t^{*}}^{t} a(\xi)g(u(\xi))\,\mathrm{d}\xi \geq0,
\quad \text{for every $t\in\mathopen{[}t^{*},t^{*}+\delta\mathclose{]}$,}
\end{equation}
and also
\begin{equation*}
v(t^{*}+\delta)=-\lambda \int_{t^{*}}^{t^{*}+\delta} a(t)g(u(t))\,\mathrm{d}t >0.
\end{equation*}
Therefore, by the first equation in \eqref{eq:system-S}, it follows that $u'(t^{*}+\delta)=h(t^{*}+\delta,v(t^{*}+\delta))>0$.
Finally, recalling \eqref{eq-v-positive}, from $u'(t)=h(t,v(t))\geq 0$ and $u'\not\equiv0$ in $\mathopen{[}t^{*},t^{*}+\delta\mathclose{]}$ (since $u'(t^{*}+\delta)>0$), we have
\begin{equation*}
\|u\|_{\infty} \geq u(t^{*}+\delta) = u(t^{*})+ \int_{t^{*}}^{t^{*}+\delta} u'(t) \,\mathrm{d}t > u(t^{*})=\|u\|_{\infty},
\end{equation*}
a contradiction.
On the other hand, if
\begin{equation*}
a\equiv 0, \quad \text{in $\mathopen{[}t^{*},t^{*}+\delta\mathclose{]}$,}
\end{equation*}
then $v'\equiv0$ in $\mathopen{[}t^{*},t^{*}+\delta\mathclose{]}$ and then $v\equiv0$ in $\mathopen{[}t^{*},t^{*}+\delta\mathclose{]}$, since $v(t^{*})=0$. And so $u'\equiv0$ in $\mathopen{[}t^{*},t^{*}+\delta\mathclose{]}$, then $u\equiv\|u\|_{\infty}$. As a consequence, let 
\begin{equation*}
t^{**}=\sup\{t\in \mathopen{]}t^{*},t^{*}+T\mathclose{]} \colon a\equiv 0 \; \text{in $\mathopen{[}t^{*},t\mathclose{]}$}\}.
\end{equation*}
Since by hypothesis $a\not\equiv0$ on $\mathbb{R}$, it follows that $t^{*}<t^{**}<+\infty$. If $t^{**}\in J_{n}$ for some $n=1,\ldots, N$, we reach the thesis by replacing $t^{*}$ with $t^{**}$. If not, there exists $\delta'>0$ such that
\begin{equation*}
a(t)\leq 0, \quad \text{for a.e.~$t\in\mathopen{[}t^{**},t^{**}+\delta'\mathclose{]}$,}
\qquad 
a \not\equiv 0, \quad \text{in~$\mathopen{[}t^{**},t^{**}+\delta'\mathclose{]}$,}
\end{equation*}
and we can repeat the previous arguments to reach a contradiction. The proof is complete.
\end{proof}

Now we provide an upper bound for the $T$-periodic solutions $(u,v)$ of \eqref{sys-tilde-S-alpha}, namely the validity of \ref{cond-H-lambda} for some $R>0$.

\begin{theorem}\label{large-solution-1}
Let $h\colon\mathbb{R}\times\mathbb{R}\to\mathbb{R}$ be a continuous function, $T$-periodic in the first variable, and satisfying \ref{cond-h-0} and \ref{cond-h-2}.
Let $a\colon\mathbb{R}\to\mathbb{R}$ be a locally integrable $T$-periodic function satisfying \ref{cond-a-star}. Let $g\colon\mathopen{[}0,+\infty\mathclose{[}\to\mathopen{[}0,+\infty\mathclose{[}$ be a continuous function satisfying \eqref{eq:cond-g-star}. 
Then, there exists $R>0$ such that \ref{cond-H-lambda} holds.
\end{theorem}

\begin{proof}
By \ref{cond-h-0}, \ref{cond-h-2}, let $h^{\pm}_{\infty}\in\mathbb{R}$ and $s_{*}^{\pm}\in\mathbb{R}$ be such that
\begin{align*}
&h(t,s) \geq h^{+}_{\infty}>0, \quad \text{for all $t\in\mathbb{R}$, for all $s>s_{*}^{+}>0$,}
\\
&h(t,s) \leq h^{-}_{\infty}<0, \quad \text{for all $t\in\mathbb{R}$, for all $s<s_{*}^{-}<0$.}
\end{align*}
First, recalling \eqref{def-gamma}, we fix $R>0$ such that
\begin{equation}\label{conds-R}
-\frac{\gamma}{4} h^{-}_{\infty} > \dfrac{R}{2} 
\quad \text{ and } \quad 
\frac{\gamma}{4} h^{+}_{\infty}>\frac{R}{2}.
\end{equation}
Let $\eta=\eta(R)>0$ be such that
\begin{equation}\label{eta-cond}
-\frac{\gamma}{4} \underline{h}(-\eta) > \dfrac{R}{2} 
\quad \text{ and } \quad 
\frac{\gamma}{4} \underline{h}(\eta) > \dfrac{R}{2}.
\end{equation}
Notice that $\eta$ exists since $\underline{h}$ is a strictly increasing function and
\begin{equation*}
\limsup_{s\to-\infty}\underline{h}(s)\leq h^{-}_{\infty},
\qquad
\liminf_{s\to+\infty}\underline{h}(s)\geq h^{+}_{\infty}.
\end{equation*}
Let
\begin{align*}
&\lambda_{1,-}^{*} = \lambda_{1,-}^{*} (R) := \inf \biggl{\{} \lambda >0 \colon -\frac{\gamma}{8}\overline{h}\biggl{(}-\lambda \underline{g}(R) \, A^{*}\biggl{(}\frac{\gamma}{8}\biggr{)}\biggr{)} > \dfrac{R}{2} \biggr{\}},
\\
&\lambda_{1,+}^{*} = \lambda_{1,+}^{*} (R) := \inf \biggl{\{} \lambda >0 \colon \frac{\gamma}{8}\overline{h}\biggl{(}\lambda \underline{g}(R) \, A^{*}\biggl{(}\frac{\gamma}{8}\biggr{)}\biggr{)} > \dfrac{R}{2} \biggr{\}},
\\
&\lambda_{2,-}^{*} = \lambda_{2,-}^{*} (R) :=\eta \biggl{[} \underline{g}(R)\, A^{*}\biggl{(}-\frac{R}{2 \overline{h}(-\eta)}\biggr{)}\biggr{]}^{-1},
\\
&\lambda_{2,+}^{*} = \lambda_{2,+}^{*} (R) :=\eta \biggl{[} \underline{g}(R)\, A^{*}\biggl{(}\frac{R}{2 \overline{h}(\eta)}\biggr{)}\biggr{]}^{-1},
\end{align*}
and define
\begin{equation*}
\lambda^{*}:=\max\bigl{\{}\lambda_{1,-}^{*},\lambda_{1,+}^{*},\lambda_{2,-}^{*},\lambda_{2,+}^{*}\bigr{\}}.
\end{equation*}
Notice that the sets in the definitions of $\lambda_{1,\pm}^{*}$ are not empty and hence the infima are non-negative real numbers. Indeed, if the set in the definition of $\lambda_{1,-}^{*}$ were empty (for the other we reason analogously), we would have
\begin{equation}\label{pdnei}
-\frac{\gamma}{8}\overline{h}\biggl{(}-\lambda \underline{g}(R)\, A^{*} \biggl{(} \dfrac{\gamma}{8}\biggr{)} \biggr{)}< \dfrac{R}{2},
\quad \text{for all $\lambda>0$.}
\end{equation}
Taking $\lambda\to+\infty$ in \eqref{pdnei} we obtain a contradiction from \eqref{conds-R} and
\begin{equation*}
\limsup_{s\to-\infty}\overline{h}(s)\leq h^{-}_{\infty},
\qquad
\liminf_{s\to+\infty}\overline{h}(s)\geq h^{+}_{\infty}.
\end{equation*}
Moreover, $\lambda_{1,\pm}^{*}\neq 0$ since 
\begin{align*}
(0,\varepsilon)\cap \biggl{\{} \lambda >0 \colon -\frac{\gamma}{8}\overline{h}\biggl{(} -\lambda \underline{g}(R)\, A^{*}\biggl{(} \dfrac{\gamma}{8}\biggr{)}\biggr{)}> \dfrac{R}{2} \biggr{\}}=\emptyset,
\\
(0,\varepsilon)\cap \biggl{\{} \lambda >0 \colon \frac{\gamma}{8}\overline{h}\biggl{(} \lambda \underline{g}(R)\, A^{*}\biggl{(} \dfrac{\gamma}{8}\biggr{)}\biggr{)}> \dfrac{R}{2} \biggr{\}}=\emptyset,
\end{align*}
for sufficiently small $\varepsilon>0$.

Let $\lambda>\lambda^{*}$, $\alpha\geq 0$, and $w$ be as in the statement.
By contradiction, we suppose that there exists a $T$-periodic solution $(u,v)$ of \eqref{syst-alpha-2} with $0\leq u(t) \leq \|u\|_{\infty}=R$ for all $t\in\mathbb{R}$.
Reasoning as in Lemma~\ref{lem-maximum} and recalling that $w\equiv0$ on $(\mathbb{R}/T\mathbb{Z}) \setminus \bigcup_{n=1}^{N}J_{n}$, we know that there exist an index $n\in\{1,\ldots,N\}$ and a point $t^{*}\in J_{n} :=\mathopen{[}\sigma_{n},\tau_{n} \mathclose{]}$ such that $u(t^{*})=\|u\|_{\infty}=R$. 

Two non-exclusive situations can occur
\begin{equation*}
t^{*}\in\Bigl{[}\sigma_{n},\tau_{n}-\frac{\gamma}{2}\Bigr{]},
\quad \text{ or } \quad
t^{*}\in\Bigl{[}\sigma_{n}+\frac{\gamma}{2},\tau_{n}\Bigr{]}.
\end{equation*}
For the first alternative, it will be used the first inequality in \eqref{eta-cond}, while for the second alternative, the second inequality in \eqref{eta-cond}. Let us consider the first alternative (the other is analogous reasoning backwards). In particular, observe that
\begin{equation*}
\int_{t^{*}}^{t^{*}+\frac{\gamma}{2}} a(t)\,\mathrm{d}t > 0.
\end{equation*}
We claim that there exists $\hat{t}\in\mathopen{[}t^{*},t^{*}+\frac{\gamma}{4}\mathclose{]}\subseteq\mathopen{[}\sigma_{n},\tau_{n}\mathclose{]}$ such that $u(\hat{t})=R/2$.
By contradiction, we suppose that
\begin{equation*}
\dfrac{R}{2} < u(t) \leq R, \quad \text{for every $t\in\Bigl{[}t^{*},t^{*}+\frac{\gamma}{4}\Bigr{]}$.}
\end{equation*}
Since $u'(t^{*})=0$ and so $v(t^{*})=0$, integrating the second equation in \eqref{syst-alpha-2} we get
\begin{align*}
v(t) 
&= v(t^{*}) - \lambda \int_{t^{*}}^{t} a(\xi) g(u(\xi)) \,\mathrm{d}\xi - \alpha \int_{t^{*}}^{t} w(\xi) \,\mathrm{d}\xi
\\
&\leq - \lambda \underline{g}(R) \int_{t^{*}}^{t} a(\xi) \mathrm{d}\xi \leq 0,
\quad \text{for all $t\in\Bigl{[}t^{*},t^{*}+\frac{\gamma}{4}\Bigr{]}$,}
\end{align*}
and, recalling \eqref{def-A-star}, we obtain
\begin{equation*}
v(t) 
\leq - \lambda \underline{g}(R) \int_{t^{*}}^{t^{*}+\frac{\gamma}{8}} a(\xi) \mathrm{d}\xi 
\leq - \lambda \underline{g}(R) A^{*}\biggl{(} \dfrac{\gamma}{8}\biggr{)},
\quad \text{for all $t\in\Bigl{[}t^{*}+\frac{\gamma}{8},t^{*}+\frac{\gamma}{4}\Bigl{]}$.}
\end{equation*}
Next, integrating the first equation in \eqref{syst-alpha-2} we have
\begin{align}
\dfrac{R}{2} &\geq \biggl{|} u\biggl{(}t^{*}+\frac{\gamma}{4}\biggr{)}-u\biggl{(}t^{*}+\frac{\gamma}{8}\biggr{)} \biggr{|}
=\biggl{|}  \int_{t^{*}+\frac{\gamma}{8}}^{t^{*}+\frac{\gamma}{4}} h(t,v(t)) \,\mathrm{d}t \biggr{|}
\\
&\geq - \int_{t^{*}+\frac{\gamma}{8}}^{t^{*}+\frac{\gamma}{4}} h \biggl{(} t,- \lambda \underline{g}(R) A^{*}\biggl{(} \dfrac{\gamma}{8}\biggr{)} \biggr{)} \,\mathrm{d}t \\
&\geq  - \dfrac{\gamma}{8} \overline{h} \biggl{(} - \lambda \underline{g}(R) A^{*}\biggl{(} \dfrac{\gamma}{8}\biggr{)} \biggr{)}.
\label{eq-contr-1}
\end{align}
Then, a contradiction is reached since $\lambda>\lambda^{*}\geq \lambda^{*}_{1,-}$.

Therefore, let $\hat{t}\in\mathopen{]}t^{*},t^{*}+\frac{\gamma}{4}\mathclose{[}$ be such that $u(\hat{t})=R/2$. Due to the fact that $u'(t)<0$ in $\mathopen{]}t^{*},\tau_{n}\mathclose{]}$ (since $v(t)<0$ therein), we have that
\begin{align*}
&\dfrac{R}{2} \leq u(t) \leq R, \quad \text{for all $t\in\mathopen{[}t^{*},\hat{t}\mathclose{]}$,}
\\
&0 \leq u(t) \leq \dfrac{R}{2}, \quad \text{for all $t\in\mathopen{]}\hat{t},\tau_{n}\mathclose{]}$.}
\end{align*}

We claim that $v(\hat{t})<-\eta$.
Assume by contradiction that $v(\hat{t})\geq-\eta$. Observing that $v(t)\geq v(\hat{t})\geq-\eta$ for all $t\in\mathopen{[}t^{*},\hat{t}\mathclose{]}$ (since $v'(t)\leq0$ therein), we have
\begin{equation*}
u'(t) = h(t,v(t)) \geq h(t,-\eta) \geq \overline{h}(-\eta), \quad \text{for all $t\in\mathopen{[}t^{*},\hat{t}\mathclose{]}$.}
\end{equation*}
Therefore, we obtain that
\begin{equation*}
-\dfrac{R}{2} = u(\hat{t})- u(t^{*}) = \int_{t^{*}}^{\hat{t}} u'(t) \,\mathrm{d}t \geq (\hat{t}-t^{*}) \overline{h}(-\eta)
\end{equation*}
and so (since $\overline{h}(-\eta)<0$)
\begin{equation*}
\hat{t}\geq t^{*} - \dfrac{R}{2 \overline{h}(-\eta)} = t^{*} + \dfrac{R}{2 |\overline{h}(-\eta)|}.
\end{equation*}
We deduce that
\begin{align}
-\eta 
&\leq v(\hat{t}) 
= v(t^{*}) - \lambda \int_{t^{*}}^{\hat{t}} a(t) g(u(t)) \,\mathrm{d}t - \alpha \int_{t^{*}}^{\hat{t}} w(t) \,\mathrm{d}t 
\\
&\leq - \lambda \underline{g}(R) A^{*}\left(\frac{R}{2 |\overline{h}(-\eta)|}\right) 
< -\eta,
\label{eq-contr-2}
\end{align}
a contradiction in the election of $\lambda$ since $\lambda>\lambda^{*}\geq \lambda^{*}_{2,-}$.
Then, an integration of the first equation in \eqref{syst-alpha-2}, the first inequality in \eqref{eta-cond}, and the fact that $\hat{t}\leq t^{*}+\frac{\gamma}{4}\leq \tau_{n}-\frac{\gamma}{4}$ yield
\begin{equation*}
u(\tau_{n}) = u(\hat{t})+\int_{\hat{t}}^{\tau_{n}} h(t,v(t))\,\mathrm{d}t 
\leq \dfrac{R}{2}+\int_{\tau_{n}-\frac{\gamma}{4}}^{\tau_{n}}h(t,v(t))\,\mathrm{d}t 
< \dfrac{R}{2} + \frac{\gamma}{4} \underline{h}(-\eta) <0,
\end{equation*}
a contradiction with the fact that $u$ is a non-negative function.
\end{proof}

As a final result, we introduce a condition at infinity ensuring the validity of \ref{cond-H-lambda-00}, so the existence of a priori bounds for every $\lambda>0$.

\begin{theorem}\label{large-solution-2}
Let $h\colon\mathbb{R}\times\mathbb{R}\to\mathbb{R}$ be a continuous function, $T$-periodic in the first variable, and satisfying \ref{cond-h-0} and \ref{cond-h-2}. 
Let $a\colon\mathbb{R}\to\mathbb{R}$ be a locally integrable $T$-periodic function satisfying \ref{cond-a-star}. 
Let $\lambda>0$. Let $g\colon\mathopen{[}0,+\infty\mathclose{[}\to\mathopen{[}0,+\infty\mathclose{[}$ be a continuous function satisfying \eqref{eq:cond-g-star} and
\begin{equation}\label{super-infinity-2}
\lim_{s\to+\infty} \mathrm{sign}(K) \frac{h(t,Kg(s))}{s} = +\infty, \quad \text{uniformly in $t$, for every $K\in\mathbb{R}\setminus\{0\}$.}
\end{equation} 
Then, \ref{cond-H-lambda-00} holds.
\end{theorem}

\begin{proof}
Let $\lambda>0$. For every $R>0$ we define $\eta=\eta(R)>0$ such that 
\begin{equation*}
-\frac{\gamma}{8} \underline{h}(-\eta) = \dfrac{R}{2} 
\quad \text{ and } \quad 
\frac{\gamma}{8} \underline{h}(\eta) = \dfrac{R}{2},
\end{equation*}
so that \eqref{eta-cond} holds. The proof follows exactly the same scheme of the one of Theorem~\ref{large-solution-1}. The only modifications are in the way we reach the contradictions in \eqref{eq-contr-1} and in \eqref{eq-contr-2}. In the present situation $\lambda$ is fixed and the contradictions are reached taking $R$ sufficiently large, exploiting 
\begin{equation}\label{super-infinity-3}
\lim_{s\to+\infty} \mathrm{sign}(K) \frac{h(t,K\underline{g}(s))}{s} = +\infty, \quad \text{uniformly in $t$, for every $K\in\mathbb{R}\setminus\{0\}$.}
\end{equation} 
To conclude, we just observe that condition \eqref{super-infinity-3} follows from \eqref{super-infinity-2}. Indeed, let $C_{s}\in \mathopen{[}s/2,s\mathclose{]}$ be such that $g(C_{s})=\underline{g}(s)$, then $h(t,K\underline{g}(s))/s = h(t,Kg(C_{s}))/C_{s} \cdot C_{s}/s$ and so
$\mathrm{sign}(K) h(t,K\underline{g}(s))/s \geq 1/2 \cdot \mathrm{sign}(K) h(t,Kg(C_{s}))/C_{s}$, 
and from this clearly \eqref{super-infinity-3} follows from \eqref{super-infinity-2}. 
We omit the other details since they require only minor modifications of the proof of Theorem~\ref{large-solution-1}.
\end{proof}

\section{Examples}\label{section-5}

In this final section, we present some applications of our existence results to second-order differential equations.

\subsection{$\phi$-Laplacian operator}\label{section-5.1}

We consider the $\phi$-Laplacian differential equation
\begin{equation}\label{eq-phi-Lapl}
(\phi(u'))'+\lambda a(t)g(u)=0,
\end{equation}
where $\phi\colon\mathbb{R}\to\mathbb{R}$ is an increasing homeomorphism with $\phi(0)=0$ (for the readers convenience, we stress that $\phi(\mathbb{R})=\mathbb{R}$).
We notice that equation \eqref{eq-phi-Lapl} can be written as a planar system of the form
\begin{equation*}
\begin{cases}
\, u'=\phi^{-1}(v), \\
\, v'=-\lambda a(t) g(u),
\end{cases}
\end{equation*}
which corresponds to system \eqref{eq:system-S} with $h(t,s)=\phi^{-1}(s)$. Thus, hypotheses \ref{cond-h-0} and \ref{cond-h-2} are trivially satisfied. 

We can now present the following existence result for $T$-periodic solutions of the $\phi$-Laplacian equation \eqref{eq-phi-Lapl}, as a direct application of Theorem~\ref{th-main}.

\begin{theorem}\label{th-5.1}
Let $\phi\colon\mathbb{R}\to\mathbb{R}$ be an increasing homeomorphism with $\phi(0)=0$ satisfying the upper $\sigma$-condition at zero, that is
\begin{equation}\label{cond-upper-sigma}
\limsup_{s\to 0^{+}} \dfrac{|\phi(\sigma s)|}{\phi(s)} < +\infty, 
\quad \text{for every $|\sigma|>1$.}
\end{equation}
Let $a\colon\mathbb{R}\to\mathbb{R}$ be a locally integrable $T$-periodic function satisfying \eqref{eq:cond-a-diesis} and \ref{cond-a-star}. Let $g\colon\mathopen{[}0,+\infty\mathclose{[}\to\mathopen{[}0,+\infty\mathclose{[}$ be a continuous function, regularly oscillating at zero, satisfying \eqref{eq:cond-g-star} and
\begin{equation}\label{cond-gphi}
\lim_{s\to 0^{+}} \dfrac{g(s)}{\phi(s)} = 0.
\end{equation}
Then, there exists $\lambda^{*}>0$ such that for every $\lambda>\lambda^{*}$, \eqref{eq-phi-Lapl} has at least one positive $T$-periodic solution.
\end{theorem}

\begin{proof}
We are going to show that the hypotheses of Theorem~\ref{th-main} are satisfied.
We prove that hypotheses \eqref{cond-upper-sigma} and \eqref{cond-gphi} implies that \eqref{main-cond-zero} holds, that is
\begin{equation}\label{eq-th-5.1-HP}
\lim_{s\to 0^{+}} \dfrac{\phi^{-1}(K g(s))}{s} = 0,
\quad \text{for every $K\in\mathbb{R}$.}
\end{equation}
The case $K=0$ is trivial since $\phi^{-1}(0)=0$. Let $K>0$ and $\varepsilon\in\mathopen{]}0,1\mathclose{[}$, then there exists some $\delta>0$ such that
\begin{equation}\label{eq-th-5.1}
0 < \dfrac{\phi(s)}{\phi(\varepsilon s)} \dfrac{g(s)}{\phi(s)} < \dfrac{1}{K}, 
\quad \text{for every $0<s<\delta$,}
\end{equation}
due to \eqref{cond-gphi} and the fact that \eqref{cond-upper-sigma} implies that
\begin{equation*}
\limsup_{s\to 0^{+}} \dfrac{\phi(s)}{|\phi(\sigma s)|} < +\infty, 
\quad \text{for every $0<|\sigma|<1$.}
\end{equation*}
Therefore, from \eqref{eq-th-5.1} we deduce that
\begin{equation*}
0 < K g(s) < \phi(\varepsilon s),
\quad \text{for every $0<s<\delta$,}
\end{equation*}
and thus, by applying $\phi^{-1}$, \eqref{eq-th-5.1-HP} and so \eqref{main-cond-zero} hold. If $K<0$, we proceed similarly by considering
\begin{equation*}
\dfrac{1}{K} < \dfrac{\phi(s)}{\phi(-\varepsilon s)} \dfrac{g(s)}{\phi(s)} < 0, 
\quad \text{for every $0<s<\delta$,}
\end{equation*}
instead of \eqref{eq-th-5.1} and repeating the same argument.
Then, the result follows by a direct application of Theorem~\ref{th-main}.
\end{proof}

The following existence result is a consequence of Theorem~\ref{th-main-2}. We omit the proof since it is analogous to the one of Theorem~\ref{th-5.1}.

\begin{theorem}\label{th-5.2}
Let $\phi\colon\mathbb{R}\to\mathbb{R}$ be an increasing homeomorphism with $\phi(0)=0$ satisfying the upper $\sigma$-condition at zero \eqref{cond-upper-sigma} and the lower $\sigma$-condition at infinity
\begin{equation}\label{cond-lower-sigma}
\limsup_{s\to \pm\infty} \dfrac{\phi(\sigma s)}{\phi(s)} <+\infty, 
\quad \text{for every $\sigma>1$.}
\end{equation}
Let $a\colon\mathbb{R}\to\mathbb{R}$ be a locally integrable $T$-periodic function satisfying \eqref{eq:cond-a-diesis} and \ref{cond-a-star}. Let $g\colon\mathopen{[}0,+\infty\mathclose{[}\to\mathopen{[}0,+\infty\mathclose{[}$ be a continuous function, regularly oscillating at zero, satisfying \eqref{eq:cond-g-star}, \eqref{cond-gphi} and
\begin{equation*}
\lim_{s\to \pm \infty} \dfrac{g(|s|)}{|\phi(s)|} = +\infty.
\end{equation*}
Then, for every $\lambda>0$, \eqref{eq-phi-Lapl} has at least one positive $T$-periodic solution.
\end{theorem}

\begin{remark}\label{rem-5.1}
The terminology for the $\sigma$-conditions at zero and at infinity are taken from \cite{GHMaZa-2011}.
Actually, the upper and lower $\sigma$-conditions at infinity were previously introduced and applied in \cite{GHMaZa-1993,GHMaZa-1996} for an odd homeomorphism $\phi$. As observed in \cite{GHMaZa-1993}, the upper $\sigma$-condition at infinity is related to the classical $\Delta_{2}$-condition considered in theory of Orlicz spaces (see \cite[Chapter~8]{AdFo-2003}); more precisely the $\Delta_{2}$-condition near infinity is expressed by the fact that $\phi(\sigma s) / \phi(s)$ is bounded from above in a neighborhood of infinity. Here we use the same kind of upper bound in a neighborhood of zero.
\hfill$\lhd$
\end{remark}

A corollary of Theorem~\ref{th-5.2} can be given for the $(p,q)$-Laplacian operator, namely
\begin{equation*}
\phi(s) = |s|^{p-2}s+|s|^{q-2}s,
\quad \text{with $1< q < p < +\infty$.}
\end{equation*}
We omit the straightforward proof.

\begin{theorem}\label{th-5.3}
Let $1\leq q < p < +\infty$. Let $a\colon\mathbb{R}\to\mathbb{R}$ be a locally integrable $T$-periodic function satisfying \eqref{eq:cond-a-diesis} and \ref{cond-a-star}. Let $g\colon\mathopen{[}0,+\infty\mathclose{[}\to\mathopen{[}0,+\infty\mathclose{[}$ be a continuous function, regularly oscillating at zero, satisfying 
\begin{equation*}
\lim_{s\to 0^{+}} \dfrac{g(s)s}{s^{q}} = 0
\quad \text{ and } \quad
\lim_{s\to +\infty} \dfrac{g(s)s}{s^{p}} = +\infty.
\end{equation*}
Then, for every $\lambda>0$, the differential equation
\begin{equation*}
(|u'|^{p-2}u'+|u'|^{q-2}u')'+\lambda a(t)g(u)=0
\end{equation*}
has at least one positive $T$-periodic solution.
\end{theorem}

\subsection{$p(t)$-Laplacian operator}\label{section-5.2}

We study the $p(t)$-Laplacian differential equation
\begin{equation}\label{eq-p-t}
(|u'|^{p(t)-2}u')'+\lambda a(t)g(u)=0,
\end{equation}
where $p\colon\mathbb{R}\to\mathopen{]}1,+\infty\mathclose{[}$ is a continuous $T$-periodic function. Then, there exist $\underline{p},\overline{p}\in\mathbb{R}$ such that
\begin{equation}\label{cond-eq-p-t}
1 < \underline{p} \leq p(t) \leq \overline{p} < +\infty, \quad \text{for all $t\in\mathbb{R}$.}
\end{equation}
Equation \eqref{eq-p-t} corresponds to the planar system
\begin{equation*}
\begin{cases}
\, |u'|^{p(t)-2}u'=v, \\
\, v'=-\lambda a(t) g(u).
\end{cases}
\end{equation*}
From the first equation we obtain
\begin{equation*}
u' = h(t,v) = |v|^{\frac{2-p(t)}{p(t)-1}} v = |v|^{\frac{1}{p(t)-1}} \mathrm{sign}(v)
\end{equation*}
and thus $h$ satisfies hypotheses \ref{cond-h-0} and \ref{cond-h-2}.

We can thus state the following result.

\begin{theorem}\label{th-5.3}
Let $p\colon\mathbb{R}\to\mathopen{]}1,+\infty\mathclose{[}$ be a continuous $T$-periodic function. Let $a\colon\mathbb{R}\to\mathbb{R}$ be a locally integrable $T$-periodic function satisfying \eqref{eq:cond-a-diesis} and \ref{cond-a-star}. Let $g\colon\mathopen{[}0,+\infty\mathclose{[}\to\mathopen{[}0,+\infty\mathclose{[}$ be a continuous function, regularly oscillating at zero, satisfying
\begin{equation}\label{cond-th5.3-0}
\limsup_{s\to 0^{+}} \dfrac{g(s)s}{s^{\overline{p}}} < +\infty
\end{equation}
and 
\begin{equation}\label{cond-th5.3-infinity}
\liminf_{s\to +\infty} \dfrac{g(s)s}{s^{\overline{p}}} >0,
\end{equation}
where $\overline{p}$ is defined as in \eqref{cond-eq-p-t}.
Then, for every $\lambda>0$, \eqref{eq-p-t} has at least one positive $T$-periodic solution.
\end{theorem}

\begin{proof}
In order to verify hypothesis \eqref{main-cond-zero} we are going to show that
\begin{equation*}
\lim_{s\to 0^{+}} \dfrac{g(s)^{\frac{1}{p(t)-1}}}{s} = 0,
\quad \text{uniformly in $t$.}
\end{equation*}
This is a consequence of \eqref{cond-th5.3-0} and
\begin{equation*}
\dfrac{g(s)^{\frac{1}{p(t)-1}}}{s} 
= \biggl{(}\dfrac{g(s)s}{s^{\overline{p}}}\biggr{)}^{\frac{1}{p(t)-1}} s^{\frac{\overline{p}-p(t)}{p(t)-1}},
\quad \text{for all $s>0$.}
\end{equation*}
Analogously, from \eqref{cond-th5.3-infinity} and the above equality, we have
\begin{equation*}
\lim_{s\to +\infty} \dfrac{g(s)^{\frac{1}{p(t)-1}}}{s} = +\infty,
\quad \text{uniformly in $t$,}
\end{equation*}
thus hypothesis \eqref{main-cond-infinity} is verified. Then, we can apply Theorem~\ref{th-main-2} to reach the thesis.
\end{proof}

\subsection{Minkowski curvature operator}\label{section-5.3}

As a last example, we investigate the positive $T$-periodic solutions of the Minkowski curvature equation
\begin{equation}\label{eq-Mink}
\biggl{(} \dfrac{u'}{\sqrt{1-(u')^{2}}} \biggr{)}'+\lambda a(t)g(u)=0,
\end{equation}
which corresponds to the planar system
\begin{equation*}
\begin{cases}
\, u'=\dfrac{v}{\sqrt{1+v^{2}}}, \\
\, v'=-\lambda a(t) g(u).
\end{cases}
\end{equation*}
In this case, $h(t,s)= s / \sqrt{1+s^{2}}$, thus \ref{cond-h-0} and \ref{cond-h-2} are trivially satisfied.

Then, we have the following corollary of Theorem~\ref{th-main}.

\begin{theorem}\label{th-5.4}
Let $a\colon\mathbb{R}\to\mathbb{R}$ be a locally integrable $T$-periodic function satisfying \eqref{eq:cond-a-diesis} and \ref{cond-a-star}. Let $g\colon\mathopen{[}0,+\infty\mathclose{[}\to\mathopen{[}0,+\infty\mathclose{[}$ be a continuous function, regularly oscillating at zero, satisfying \eqref{eq:cond-g-star} and 
\begin{equation}\label{eq-cond-mink}
\lim_{s\to 0^{+}} \dfrac{g(s)}{s} = 0.
\end{equation}
Then, there exists $\lambda^{*}>0$ such that for every $\lambda>\lambda^{*}$ equation \eqref{eq-Mink} has at least one positive $T$-periodic solution.
\end{theorem}

\begin{proof}
We notice that hypothesis \eqref{eq-cond-mink} implies that 
\begin{equation*}
\lim_{s\to 0^{+}} \dfrac{K g(s)}{s\sqrt{1+(Kg(s))^{2}}} = 0,
\quad \text{for every $K\in\mathbb{R}$,}
\end{equation*}
which is exactly \eqref{main-cond-zero}. 
Then, Theorem~\ref{th-main} applies and the proof is concluded.
\end{proof}

\appendix
\section{Maximum principles for planar systems}\label{appendix-A}

In this section, we present some maximum principles for the planar system
\begin{equation}\label{syst-max-p}
\begin{cases}
\, u' = h(t,v),
\\
\, v' = k(t,u),
\end{cases}
\end{equation}
where $h\colon\mathbb{R}\times\mathbb{R}\to\mathbb{R}$ is a continuous function, $T$-periodic in the first variable, and such that
\begin{enumerate}[leftmargin=25pt,labelsep=10pt,label=\textup{$(i)$}]
\item $h(t,0)=0$ for every $t\in\mathbb{R}$;
\label{cond-app-i}
\end{enumerate}
\begin{enumerate}[leftmargin=25pt,labelsep=10pt,label=\textup{$(ii)$}]
\item there exist $\underline{h},\overline{h} \colon \mathbb{R} \to \mathbb{R}$ continuous, with $\overline{h}$ monotone increasing, and such that $0 \leq \underline{h}(s)s \leq h(t,s)s \leq \overline{h}(s)s$, for almost every $t\in \mathbb{R}$ and for all $|s|\leq \eta$, for some $\eta>0$;
\label{cond-app-iii}
\end{enumerate}
Let $\underline{H}(s):=\int_{0}^{s}\underline{h}(\xi) \,\mathrm{d}\xi
$ and denote by $\underline{H}^{-1}_{\mathrm{l}}$ and $\underline{H}^{-1}_{\mathrm{r}}$ the left and right inverse of $\underline{H}$, respectively.

We assume that $k\colon\mathbb{R}\times\mathbb{R}\to\mathbb{R}$ is an $L^{1}$-Carath\'{e}odory function, $T$-periodic in the first variable.

We first present a weak maximum principle.

\begin{proposition}[Weak maximum principle]\label{weak-max-p}
Let $h\colon\mathbb{R}\times\mathbb{R}\to\mathbb{R}$ be a continuous function, $T$-periodic in the first variable, and satisfying \ref{cond-app-i}. Let $k\colon\mathbb{R}\times\mathbb{R}\to\mathbb{R}$ be an $L^{1}$-Carath\'{e}odory function, $T$-periodic in the first variable, and such that
\begin{equation*}
k(t,s)<0, \quad \text{for a.e.~$t\in\mathbb{R}$, for all $s\in\mathopen{]}-\infty,0\mathclose{[}$.}
\end{equation*}
If $(u,v)$ is a $T$-periodic solution of \eqref{syst-max-p}, then $u(t)\geq0$ for all $t\in\mathbb{R}$.
\end{proposition}

\begin{proof}
Let $(u,v)$ be a $T$-periodic solution of \eqref{syst-max-p}. By contradiction, we suppose that there exists $t^{*}\in\mathbb{R}$ such that $u(t^{*})<0$. First of all, we observe that, if $u(t)<0$ for all $t\in\mathbb{R}$, then
\begin{equation*}
0 = v(T)-v(0) = \int_{0}^{T} v'(t) \,\mathrm{d}t 
= \int_{0}^{T} k(t,u(t)) \,\mathrm{d}t
< 0,
\end{equation*}
which is a contradiction. Therefore, $u$ is non-negative in some points.
Let $\mathopen{]}t_{0},t_{1}\mathclose{[}\subseteq \mathbb{R}$ be the maximal interval containing $t^{*}$ such that $u(t)<0$ for all $t\in\mathopen{]}t_{0},t_{1}\mathclose{[}$.
By the $T$-periodicity, we observe that the interval is bounded and, by the continuity of $u$ and $u'$, we have $u(t_{0})=u(t_{1})=0$ and $u'(t_{0}) \leq 0 \leq u'(t_{1})$.
From the first equation in \eqref{syst-max-p} we deduce that $h(t_{0},v(t_{0}))=u'(t_{0})\leq 0$ and $h(t_{1},v(t_{1}))=u'(t_{1})\geq 0$, therefore by \ref{cond-app-i} we have $v(t_{0}) \leq 0 \leq v(t_{1})$. As a consequence, we have
\begin{equation*}
0 \leq v(t_{1})-v(t_{0}) = \int_{t_{0}}^{t_{1}} v'(t) \,\mathrm{d}t 
= \int_{t_{0}}^{t_{1}} k(t,u(t)) \,\mathrm{d}t
< 0,
\end{equation*}
a contradiction. The proof is complete.
\end{proof}

Secondly we present the following strong maximum principle.

\begin{proposition}[Strong maximum principle]\label{strong-max-time-maps}
Let $h\colon\mathbb{R}\times\mathbb{R}\to\mathbb{R}$ be a continuous function, $T$-periodic in the first variable, and satisfying \ref{cond-app-i} and \ref{cond-app-iii}. Let $k\colon\mathbb{R}\times\mathbb{R}\to\mathbb{R}$ be an $L^{1}$-Carath\'{e}odory function, $T$-periodic in the first variable. Assume that
\begin{enumerate}[leftmargin=27pt,labelsep=10pt,label=\textup{$(iii)$}]
\item there exists $\overline{k}\colon\mathopen{[}0,\gamma\mathclose{]}\to\mathopen{[}0,+\infty\mathclose{[}$ continuous and such that $0 \leq k(t,s)\leq \overline{k}(s)$, for almost every $t\in \mathbb{R}$ and for all $s\in\mathopen{[}0,\gamma\mathclose{]}$, for some $\gamma>0$, and moreover
$\overline{K}(u):=\int_{0}^{u}\overline{k}(s) \,\mathrm{d}s>0$, for all $u\in\mathopen{]}0,\gamma\mathclose{]}$.
\label{cond-app-iv}
\end{enumerate}
If at least one of the integrals
\begin{equation*}
\int_{0}^{\varepsilon}\frac{\mathrm{d}u}{\overline{h}(\underline{H}^{-1}_{\mathrm{l}}(\overline{K}(u)))}, 
\quad
\int_{0}^{\varepsilon}\frac{\mathrm{d}u}{\overline{h}(\underline{H}^{-1}_{\mathrm{r}}(\overline{K}(u)))}
\end{equation*}
diverges (for every $\varepsilon>0$ sufficiently small), then every non-trivial $T$-periodic solution $(u,v)$ of \eqref{syst-max-p} with $u(t)\geq 0$ for all $t\in\mathbb{R}$, satisfies $u(t)>0$ for all $t\in\mathbb{R}$.
\end{proposition}

Observe that the function $\underline{H}$ is strictly decreasing on $\mathopen{[}-\eta,0\mathclose{]}$ and strictly increasing on $\mathopen{[}0,\eta\mathclose{]}$, then for $c:=\min\{\underline{H}(-\eta),\underline{H}(\eta)\}$ the left and right inverse $\underline{H}^{-1}_{\mathrm{l}}$ and $\underline{H}^{-1}_{\mathrm{r}}$ are defined on $\mathopen{[}0,c\mathclose{]}$ with range contained in $\mathopen{[}-\eta,0\mathclose{]}$ and in $\mathopen{[}0,\eta\mathclose{]}$, respectively.
Hence, for $\varepsilon\leq\gamma$ sufficiently small such that $\overline{K}(\varepsilon)\leq c$, we have 
\begin{equation}\label{cond-integrals}
\overline{h}(\underline{H}^{-1}_{\mathrm{l}}(\overline{K}(s)))<0<\overline{h}(\underline{H}^{-1}_{\mathrm{r}}(\overline{K}(s))),
\quad \text{for all $s\in\mathopen{]}0,\varepsilon\mathclose{]}$}.
\end{equation}

\begin{proof}
By contradiction, we suppose that $(u,v)$ is a non-trivial $T$-periodic solution of \eqref{syst-max-p} with $u(t)\geq 0$ for all $t\in\mathbb{R}$ and such that $u$ vanishes somewhere.
Let $t_{0}\in \mathbb{R}$ be such that $u(t_{0})=0$. Without loss of generality, by the periodicity of $u$, we can suppose that $t_{0}$ is such that $u(t)>0$ in $\mathopen{]}t_{0},t_{0}+\delta\mathclose{]}$ or in $\mathopen{[}t_{0}-\delta,t_{0}\mathclose{[}$, for some $\delta>0$ with $u(t_{0}\pm\delta)\leq\varepsilon$ and $|v(t_{0}\pm\delta)|\leq\eta$. 

Assume that the first situation occurs. We observe that $u'(t_{0})=0$ (since $u$ is non-negative), $v(t_{0})=0$ (by \ref{cond-app-i}, \ref{cond-app-iii}, and the first equation in \eqref{syst-max-p}). Moreover, $v'(t)=k(t,u(t))\geq0$ for a.e.~$t\in \mathopen{[}t_{0},t_{0}+\delta\mathclose{]}$ and so $0\leq v(t)\leq v(t_{0}+\delta)$ for all~$t\in \mathopen{[}t_{0},t_{0}+\delta\mathclose{]}$. Furthermore, recalling hypothesis \ref{cond-app-iv}, we deduce that
\begin{align*}
\underline{h}(v(t)) v'(t) - \overline{k}(u(t)) u'(t)
&= \underline{h}(v(t)) k(t, u(t)) - \overline{k}(u(t)) h(t,v(t))
\\
&\leq - \overline{k}(u(t)) \bigl{(} h(t,v(t)) - \underline{h}(v(t)) \bigr{)} \leq 0,
\end{align*}
for almost every $t\in\mathopen{[}t_{0},t_{0}+\delta\mathclose{]}$, where the last inequality follows from \ref{cond-app-iii}.
Therefore, we have
\begin{equation*}
\dfrac{\mathrm{d}}{\mathrm{d}t} \bigl{(} \underline{H}(v(t)) - \overline{K}(u(t)) \bigr{)} \leq 0,
\quad \text{for a.e.~$t\in\mathopen{[}t_{0},t_{0}+\delta\mathclose{]}$,}
\end{equation*}
and so
\begin{equation*}
\underline{H}(v(t)) - \overline{K}(u(t)) \leq \underline{H}(v(t_{0})) - \overline{K}(u(t_{0})) = 0, 
\quad \text{for all $t\in\mathopen{[}t_{0},t_{0}+\delta\mathclose{]}$.}
\end{equation*}
Since $\underline{H}$ is strictly increasing on a right neighborhood of zero, we find
\begin{equation*}
0 \leq v(t) \leq \underline{H}^{-1}_{\mathrm{r}}(\overline{K}(u(t))),
\quad \text{for all $t\in\mathopen{[}t_{0},t_{0}+\delta\mathclose{]}$,}
\end{equation*}
and hence, by the monotonicity of $\overline{h}$, we find
\begin{equation*}
u'(t) = h(t,v(t)) \leq \overline{h}(v(t)) \leq \overline{h}(\underline{H}^{-1}(\overline{K}(u(t)))),
\quad \text{for all $t\in\mathopen{[}t_{0},t_{0}+\delta\mathclose{]}$.}
\end{equation*}
Observe that $\overline{h}(\underline{H}^{-1}(\overline{K}(u(t))))>0$ on $\mathopen{]}t_{0},t_{0}+\delta\mathclose{]}$ since $u(t)>0$ in the same interval.
Next, dividing by $\overline{h}(\underline{H}^{-1}(\overline{K}(u(t))))$, an integration leads to
\begin{equation*}
\int_{t}^{t_{0}+\delta} \dfrac{u'(\xi)}{\overline{h}(\underline{H}^{-1}(\overline{K}(u(\xi))))} \mathrm{d}\xi
\leq t_{0}+\delta-t,
\quad \text{for all $t\in\mathopen{]}t_{0},t_{0}+\delta\mathclose{]}$,}
\end{equation*}
and, by a change of variable, we have
\begin{equation*}
\int_{u(t)}^{u(t_{0}+\delta)} \dfrac{\mathrm{d}s}{\overline{h}(\underline{H}^{-1}(\overline{K}(s)))} 
\leq t_{0}+\delta-t,
\quad \text{for all $t\in\mathopen{]}t_{0},t_{0}+\delta\mathclose{]}$.}
\end{equation*}
At last, passing to the limit as $t\to (t_{0})^{+}$, we obtain
\begin{equation*}
\int_{0}^{u(t_{0}+\delta)} \dfrac{\mathrm{d}s}{\overline{h}(\underline{H}^{-1}(\overline{K}(s)))} 
\leq \delta,
\end{equation*}
a contradiction with respect to the hypothesis of the divergence of the integral at $0^{+}$.

If $u(t)>0$ in $\mathopen{[}t_{0}-\delta,t_{0}\mathclose{[}$, we reach a contradiction following a similar argument. The proof is complete.
\end{proof}

\begin{remark}\label{rem-A.1}
From an inspection of the proof, it is apparent the fact that if we assume by contradiction that a non-trivial $T$-periodic solution $(u,v)$ with $u\geq 0$ is such that $u$ vanishes at some point, then there are at least two points $t_{0},t_1$ such that $u(t)>0$ in $\mathopen{]}t_{0},t_{0}+\delta_{0}\mathclose{]}$ and $u(t)>0$ in $\mathopen{]}t_{1}-\delta_{1},t_{1}\mathclose{]}$, for some $\delta_{0},\delta_{1}>0$. Hence, the divergence of at least one of the two improper integrals in \eqref{cond-integrals} is sufficient to achieve a contradiction.
This is not the case for other boundary value problems, for instance the Neumann one, where the divergence of both improper integrals would be required.
\hfill$\lhd$
\end{remark}

\begin{remark}\label{rem-A.2}
The condition expressed in \eqref{cond-integrals} is sharp, as one can see from an analysis of the autonomous planar system
\begin{equation*}
\begin{cases}
\, u' = h(v),
\\
\, v'=k(u),
\end{cases}
\end{equation*}
with $h\colon\mathbb{R}\to\mathbb{R}$ a strictly increasing continuous function such that $h(0)=0$, and $k\colon\mathopen{[}0,+\infty\mathclose{[}\to\mathopen{[}0,+\infty\mathclose{[}$ a continuous function such that $k(0)=0$ and $k(s)>0$ for $s>0$. We observe that $h=\overline{h}=\underline{h}$ and $k=\overline{k}=\underline{k}$.
In this case, we have a Hamiltonian system with a geometry of saddle point at the origin and the level line $H(v) - K(u)=0$ splits as the union of the origin (equilibrium point), a stable manifold entering in the fourth quadrant and an unstable one in the first quadrant. In this case, the divergence of the integrals means that the solutions on the two manifolds do not hit the origin in finite time. The above autonomous system is related to the quasilinear second-order equation
\begin{equation*}
(\phi(u'))' = k(u)
\end{equation*}
for $\phi=h^{-1}$. In this case our condition turns out to be equivalent to the one involving the Legendre transform in \cite[Theorem~1.1.1]{PuSe-2007}, that is
\begin{equation*}
\int_{0^{+}} \dfrac{\mathrm{d}s}{\mathcal{H}^{-1}(K(s))} = \infty,
\end{equation*}
where $\mathcal{H}(s)=s\phi(s)-\int_{0}^{s} \phi(\xi)\,\mathrm{d}\xi$ (with the additional assumption in \cite{PuSe-2007} of monotonicity of $k$ in a right neighborhood of zero, an assumption which is not required in Proposition~\ref{strong-max-time-maps}).
\hfill$\lhd$
\end{remark}

\bibliographystyle{elsart-num-sort}
\bibliography{FeSPZa-biblio}

\end{document}